\documentclass[reqno]{amsart}
\usepackage{amssymb,amscd,verbatim, amsthm,graphicx, color}
\usepackage[mathscr]{eucal}

\usepackage{showlabels}
\usepackage{longtable}

\newcommand \fk[1]{{{\mathfrak #1}}}
\newcommand \C[1]{{\mathcal #1}}
\newcommand \ovl[1]{{\overline {#1}}}

\newcommand \sbst[2]{{\substack{{#1}\\ {#2}}}}

\newcommand \wti[1]{{\widetilde {#1}}}

\newcommand\fg{\mathfrak g}

\newcommand \bA{{\mathbb A}}
\newcommand \bC{{\mathbb C}}

\newcommand \bH{{\mathbb H}}
\newcommand \bR{{\mathbb R}}
\newcommand \bZ{{\mathbb Z}}

\newcommand\one{1\!\!1}

\newcommand\CO{{\C O}}

\newcommand\ep{{\epsilon}}

\newcommand\om{{\omega}}
\newcommand\al{{\alpha}}
\newcommand\sig{{\sigma}}

\newcommand\fh{{\mathfrak h}}

\newcommand\LR{\Leftrightarrow}

\newcommand\sem{\mathsf{ss}}
\newcommand\sgn{\mathsf{sgn}}
\newcommand\triv{\mathsf{triv}}
\newcommand\gr{\mathsf{gr}}
\newcommand\refl{\mathsf{refl}}

\newtheorem{proposition}{Proposition}[subsection]
\newtheorem{corollary}[proposition]{Corollary}
\newtheorem{lemma}[proposition]{Lemma}
\newtheorem{theorem}[proposition]{Theorem}
\newtheorem{conjecture}[proposition]{Conjecture}

\theoremstyle{definition}
\newtheorem{definition}[proposition]{Definition}
\newtheorem{remark}[proposition]{Remark}

\newtheorem{example}[proposition]{Example}

\newcommand\Hom{\operatorname{Hom}}
\newcommand\Ind{\operatorname{Ind}}

\newcommand\Id{\operatorname{Id}}
\newcommand\Ad{\operatorname{Ad}}

\newcommand\re{\operatorname{Re}}

\newcommand\cc{\textsf{cc}}
\newcommand\LWT{\textsf{LWT}}
\newcommand\St{\textsf{St}}
\newcommand\reg{\textsf{reg}}

\usepackage[mathscr]{eucal}

\newcommand{\Irr}{\mathsf{Irr}}

\numberwithin{equation}{subsection}

\newfont{\lge}{cmmi10 scaled 1640}

\newcommand \Caa{\lge{a}}

\newcommand \Ca[1]{{\mathcal #1}}

\begin{document}



\bigskip
\title{Hermitian forms for affine Hecke algebras} 
\author{Dan Barbasch}
       \address[D. Barbasch]{Dept. of Mathematics\\
               Cornell University\\Ithaca, NY 14850}
       \email{barbasch@math.cornell.edu}

\author{Dan Ciubotaru}
        \address[D. Ciubotaru]{Mathematical Institute\\ University of
          Oxford\\Andrew Wiles Building\\Oxford, OX2 6GG, UK}
        \email{dan.ciubotaru@maths.ox.ac.uk}

\begin{abstract}We study invariant hermitian forms for
  finite dimensional 
  modules over (graded) affine
  Hecke algebras with a view towards a unitarity algorithm.

\end{abstract}

\thanks{The first author was partially supported by NSF-DMS grants 0967386, 0901104, and an NSA grant. The second author was partially supported by NSF-DMS grants 0968065, 1302122, and NSA-AMS 111016.}

\maketitle

\bigskip
\setcounter{tocdepth}{1}
\tableofcontents



\section{Introduction}\label{sec:1}
In this paper, we study  invariant hermitian forms for the (graded) affine Hecke 
algebras that appear in the theory of reductive $p$-adic groups with a view towards a unitarity algorithm in the $p$-adic setting, analogous to the algorithm \cite{ALTV} for real reductive groups. We explain the main results next.

\subsection{} In \cite{BC-sanya}, we classified the star operations (conjugate-linear
involutive anti-automorphims) for the graded affine
  Hecke algebra $\bH$ with unequal parameters which preserve a natural
  filtration of $\bH$ (section \ref{sec:2}). This can be viewed as an analogue of the
  problem of classifying the star operations for the enveloping
  algebra $U(\fg)$ of a complex semisimple Lie algebra which preserve
  $\fg.$  It turns out that there are only two such star operations: $*$
  and $\bullet$, Definition \ref{d:basicstars}. 

The anti-automorphism
  $\star$ is known to correspond to the natural star operation of the
  Hecke algebra of a reductive $p$-adic group, i.e., $f^*(g)=\overline
  {f(g^{-1})}$, see \cite{BM1,BM2}. 
On the other hand, the anti-automorphism $\bullet$ is the Hecke
algebra analogue of the ``compact star operation'' for $(\fg,
K)$-modules studied by Adams-van Leeuwen-Trapa-Vogan \cite{ALTV} and Yee \cite{Y}.

With this motivation, we investigate the basic properties of the signature of
$\bullet$-invariant hermitian forms for finite dimensional
$\bH$-modules (sections \ref{sec:herm-forms}--\ref{sec:Jantzen-filtration}). We
prove that every irreducible $\bH$-module with real central character
admits a nondegenerate $\bullet$-invariant hermitian form, Corollary \ref{c:bullettemp}, and
moreover, when $\bH$ is of geometric type, this form  can be normalized canonically so that it is
positive definite on every isotypic component of a lowest
$W$-type, Corollary \ref{c:lwt}. For the first claim, we explicitly determine in Theorem
\ref{t:herm-dual} the $\bullet$-hermitian
dual of any given simple $\bH$-module, in terms of the Langlands
datum, and we exhibit in Proposition \ref{p:induced-bulletform} an
explicit invariant hermitian form. The second claim follows by
comparing the Langlands classification with the geometric
classification of simple and standard $\bH$-module \cite{L2}, together
with an argument involving the ``signature at infinity'' of the form.

These results represent the Hecke algebra analogue of the similar
results about c-invariant forms of $(\fg,K)$-modules
\cite{ALTV}. Motivated by the algorithm of \cite{ALTV} (see also \cite{Vo}),  we define in
section \ref{sec:Jantzen-filtration} hermitian Kazhdan-Lusztig
polynomials (see Definition \ref{e:herm-KL}). This definition is based on an analysis of the Jantzen filtration. We conjecture (Conjecture
\ref{main-conj}) a simple
relation with the Kazhdan-Lusztig polynomials for graded Hecke algebras \cite{L3}. In the remainder of section \ref{sec:Jantzen-filtration},
we offer some evidence for this conjecture, by analyzing the
case of regular central character, and the interesting examples of the
subregular central character in types $B_2$ and $G_2.$

\subsection{}
The algebra $\bH$ has a large abelian subalgebra $\bA.$ Since
$\bullet$ preserves $\bA$ (unlike the classical $\star$), it is
interesting to consider the weight spaces for $\bA$ and study
signatures of forms in this way. We prove a number of results along
these lines, for example, a linear independence result for
$\bA$-characters of irreducible $\bH$-modules, Theorem
\ref{t:A-linindep}, as well as explicit formulas for the
$\bullet$-forms when the $\bA$-parameter is sufficiently dominant,
e.g., Corollary \ref{c:reg-form}. These results enter in an essential way in our proof of Conjecture \ref{main-conj} in the regular case.

\section{Preliminaries: star operations}\label{sec:2}

\subsection{Graded affine Hecke algebra}

We fix an $\bR$-root system $\Phi=(V,R,V^\vee, R^\vee)$. This means
that $V, V^\vee$ are finite dimensional $\bR$-vector spaces, with a
perfect bilinear pairing $(~,~): V\times V^\vee\to \bR$, where $R\subset
V\setminus\{0\},$ $R^\vee\subset V^\vee\setminus\{0\}$ are finite
subsets in bijection 
\begin{equation}
R\longleftrightarrow R^\vee,\ \al\longleftrightarrow\al^\vee,\
\text{ satisfying }(\al,\al^\vee)=2. 
\end{equation}
Moreover, the reflections
\begin{equation}
s_\al: V\to V,\ s_\al(v)=v-(v,\al^\vee)\al, \quad s_\al:V^\vee\to
V^\vee,\ s_\al(v')=v'-(\al,v')\al^\vee, \quad \al\in R, 
\end{equation}
leave $R$ and $R^\vee$ invariant, respectively. Let $W$ be the
subgroup of $GL(V)$ (respectively $GL(V^\vee)$) generated by
$\{s_\al:~\al\in R\}$. 
We  assume that the root system $\Phi$ is reduced, meaning that
$\al\in R$ implies $2\al\notin R$. 
We fix a choice of simple roots $\Pi\subset R$, and consequently,
positive roots $R^+$ and positive coroots $R^{\vee,+}.$ Often, we will
write $\al>0$ or $\al<0$ in place of $\al\in R^+$ or $\al\in (-R^+)$,
respectively. The complexifications of $V$ and $V^\vee$ are denoted by
$V_\bC$ and $V^\vee_\bC$, respectively, and we denote by $\bar{\ }$
the complex conjugations of $V_\bC$ and $V^\vee_\bC$ induced by $V$
and $V^\vee$, respectively. Notice that 
\begin{equation}
\overline{(v,u)}=(\overline v, \overline{u}),\text{ for all }v\in
V_\bC,\ u\in V_\bC^\vee. 
\end{equation}
Let $k: \Pi\to \bR$ be a function such that $k_\al=k_{\al'}$ whenever
$\al,\al'\in \Pi$ are $W$-conjugate. Let $\bC[W]$ denote the group
algebra of $W$ and $S(V_\bC)$ the symmetric algebra over $V_\bC.$ The
group $W$ acts on $S(V_\bC)$ by extending the action on $V.$ For every
$\al\in \Pi,$  denote the difference operator by
\begin{equation}\label{e:diffop}
\Delta: S(V_\bC)\to S(V_\bC),\quad
\Delta_\al(a)=\frac{a-s_\al(a)}{\al},\text{ for all }a\in S(V_\bC).
\end{equation}

\begin{definition}\label{d:graded}
The graded affine Hecke algebra $\bH=\bH(\Phi,k)$  is the unique
associative unital algebra generated by $\bA=S(V_\bC)$ and
$\{t_w: w\in W\}$ such that  
\begin{enumerate}
\item[(i)] the assignment $t_wa\mapsto w\otimes a$ gives an
  isomorphism $\bH\cong \bC[W]\otimes S(V_\bC)$ of
  $(\bC[W],S(V_\bC))$-bimodules; 
\item[(ii)] $a t_{s_\al}=t_{s_\al}s_\al(a)+k_\al \Delta_\al(a),$
  for all $\al\in \Pi$, $a\in S(V_\bC).$
\end{enumerate}
\end{definition}

{The center of $\bH$ is $S(V_\bC)^W$ (\cite{L1}). By Schur's Lemma, the center of $\bH$
acts by scalars on each irreducible $\bH$-module. The central
characters are parameterized by $W$-orbits in $V_\bC^\vee.$ If $X$ is
an irreducible $\bH$-module, denote by $\cc(X)\in W\backslash
V_\bC^\vee$ 
its central character. By abuse of notation, we may also denote by
$\cc(X)$ a representative in $V_\bC^\vee$ of the central character of
$X$. 

If $(\pi,X)$ is a finite dimensional $\bH$-module and $\lambda\in V_\bC^\vee$, denote 
\begin{equation}
X_\lambda=\{x\in X: \text{ for every }a\in S(V_\bC),\
(\pi(a)-(a,\lambda))^nx=0,\text{ for some }n\in \mathbb N\}.
\end{equation}
If $X_\lambda\neq 0$, call $\lambda$ an $\bA$-weight of $X$. 
Let $\Omega(X)\subset V_\bC^\vee$ denote the set of
$\bA$-weights of $X$. If $X$ has a central character, it is easy to see that $\Omega(X)\subset
W\cdot \cc(X).$

\begin{definition}[Casselman's criterion]\label{d:tempered} 
Set 
$$V^+=\{\om\in V: (\om,\al^\vee)>0,\text{ for all }\al\in\Pi\}.$$
An
  irreducible $\bH$-module $X$ is called tempered if 
$$(\om,\re\lambda)\le
  0,\text{ for all }\lambda\in\Omega(X)\text{ and all }\om\in V^+.$$ 
A tempered module is
  called a discrete series module if all the inequalities are strict.
\end{definition}
}
\begin{example}Suppose that the root system $\Phi$ is semisimple. Then $\bH$ has a particular
discrete series module, the Steinberg module $\St$. This is a
one-dimensional module, on which $W$ acts via the $\sgn$
representation, and the only $\bA$-weight is $-\sum_{\al\in\Pi} k_\al
\om^\vee_\al,$ where $\om^\vee_\al$ is the fundamental coweight
corresponding to $\al.$
\end{example}

Let $w_0$ denote the long Weyl group element. The assignment
\begin{equation}\label{autom}
\delta(t_w)=t_{w_0 w w_0},\ w\in W,\quad \delta(\omega)=-w_0(\omega),\
\omega\in V_\bC.
\end{equation}
 extends to an involutive automorphism of $\bH$ when $k_{\delta(\al)}=k_\al,$ for all $\al\in\Pi.$ Clearly, when $w_0$ is central in $W$, $\delta=\Id$.

\subsection{Star operations} 

\begin{definition}
Let $\kappa:\bH\to \bH$ be a conjugate linear involutive algebra
anti-automorphism. An $\bH$-module $(\pi,X)$ is said to be
$\kappa$-hermitian if $X$ has a hermitian form $(~,~)$ which is
$\kappa$-invariant, i.e., 
$$
(\pi(h)x,y)=(x,\pi(\kappa(h))y),\quad x,y\in X,\  h\in\bH.
$$
A hermitian module $X$ is $\kappa$-unitary if the $\kappa$-hermitian
form is positive definite. 
\end{definition}

\begin{definition}\label{d:basicstars}
Define 
\begin{align}\label{e:star}
t_{w}^\star=t_{w^{-1}},\ w\in W,\quad \omega^\star=-t_{w_0}\cdot
\overline{w_0(\omega)}\cdot t_{w_0},\ \omega\in V_\bC, 
\end{align}
and
\begin{align}\label{e:bullet}
t_{w}^\bullet=t_{w^{-1}},\ w\in W,\quad \omega^\bullet=\overline\omega,\ \omega\in V_\bC.
\end{align}
It is straight-forward to check that the operations $\star$ and $\bullet$ defined in (\ref{e:star}) and (\ref{e:bullet}), respectively, extend to conjugate linear algebra anti-involutions of $\bH$.
Also, notice that 
\begin{equation}\label{e:relstar}
h^\star=t_{w_0} \cdot \delta(h)^\bullet\cdot t_{w_0},\quad h\in\bH.
\end{equation}
In particular, when $w_0$ is central in $W$, they are inner conjugate
to each other.
\end{definition}

\subsection{Classification of involutions} We define a filtration of
$\bH$ given by the degree in $S(V_\bC).$ Set $\deg t_w
a=\deg_{S(V_\bC)} a$ for every $w\in W$, and homogeneous element $a\in
S(V_\bC)$ and $F_i\bH=\text{span}\{h\in \bH: \deg h\le i\}.$ In
particular, $F_0\bH=\bC[W]$. Set $F_{-1}\bH=0.$ It is immediate from
Definition \ref{d:graded} that the associated graded algebra
$\overline\bH=\oplus_{i\ge 0} \overline \bH^i,$ where
$\overline\bH^i=F_i\bH/F_{i-1}\bH$, is naturally isomorphic to the graded Hecke
algebra for the parameter function $k_\al\equiv 0.$ 

\begin{definition}[{\cite[Definition 3.4.1]{BC-sanya}}]\label{d:admissible}
An automorphism (respectively, anti-automorphism) $\kappa$ of $\bH$ is called
\textit{filtered} if 
  $\kappa(F_i\bH)\subset F_i\bH,$ for all $i\ge 0.$ Notice that by
Definition \ref{d:graded}, this is equivalent with the requirement
that $\kappa(F_i\bH)\subset F_i\bH$ for $i=0,1.$ If, in addition,
$\kappa(t_w)=t_{w}$ (resp., $\kappa(t_w)=t_{w^{-1}}$), we say that
$\kappa$ is \textit{admissible}.  
\end{definition}

The classification of admissible star operations of $\bH$ is as follows. 

\begin{proposition}[{\cite[Proposition 3.4.3]{BC-sanya}}]\label{p:classification} Assume the root system
  $\Phi$ is simple.  If $\kappa$ is an admissible involutive star operation (in the sense of Definition \ref{d:admissible}),
then $\kappa$ is one of the operations $\star$ or
$\bullet$ from Definition \ref{d:basicstars}. 
\end{proposition}

\section{Invariant hermitian forms}\label{sec:herm-forms}
In this section, we study invariant hermitian forms for $\bH$-modules
with respect to the two star operations $\bullet$ and $\star$ from
section \ref{sec:2}. The main result is the explicit construction of $\bullet$-invariant hermitian forms for parabolically induced modules, see Proposition \ref{p:induced-bulletform} and Theorem \ref{t:induced-bulletform}.

\subsection{Relation between the forms}\label{sec:star-bullet} The relation between $\bullet$ and $\star$ from (\ref{e:relstar}) reflects into a relation between the invariant hermitian forms, when they exist, on a given simple $\bH$-module $X$. This relation is more easily expressed in terms of the extended Hecke algebra $\bH'$-modules. 

\begin{lemma}\label{l:star-bullet}
An $\bH'$-module $(\pi,X)$ admits a $\bullet$-invariant form $\langle~,~\rangle_\bullet$ if and only if it admits a $\star$-invariant form $\langle~,~\rangle_\star$. In this case, the forms are related by
\begin{equation}\label{star-bullet}
\langle v_1,v_2\rangle_\star=\langle v_1,\pi(t_{w_0}\delta) v_2\rangle_\bullet.
\end{equation}
\end{lemma}

\begin{proof}
Suppose $\langle~,~\rangle_\bullet$ exists on $X$. We verify that the $\star$-form from (\ref{star-bullet}) is indeed invariant. For $h\in \bH$, we use (\ref{e:relstar}):
\begin{equation}
\begin{aligned}
\langle \pi(h)v_1,v_2\rangle_\star&=\langle\pi(h) v_1,\pi(w_0\delta) v_2\rangle_\bullet=\langle v_1,\pi(h^\bullet t_{w_0}\delta) v_2\rangle_\bullet\\
&=\langle v_1,\pi(t_{w_0}\delta(h^*)\delta) v_2\rangle_\bullet=\langle v_1,\pi(t_{w_0}\delta)\pi( h^*) v_2\rangle_\bullet\\
&=\langle v_1,\pi(h^*)v_2\rangle_\star.
\end{aligned}
\end{equation}
The invariance under $\pi(\delta)$ is immediate since $\delta^*=\delta$ and $\delta$ commutes with $t_{w_0}.$
\end{proof}

Suppose $(\pi,X)$ is a simple $\bH$-module. Define the $\delta$-twist of $X$ to be $(\pi^\delta,X^\delta)$, where $X^\delta=X$ as vector spaces and $\pi^\delta(h)=\pi(\delta(h)).$ Suppose $X$ admits a $\bullet$-invariant form. Then, as in Lemma \ref{l:star-bullet}, we get a $\star$-invariant pairing between $X^\delta$ and $X$ via
\begin{equation}\label{star-bullet-1}
\langle~,~\rangle_\star: X^\delta\times X\to \bC,\ \langle u,v\rangle_\star=\langle u, \pi(t_{w_0})v\rangle_\bullet,\ u\in X^\delta, v\in X.
\end{equation}
This implies that, under the hypotheses, $X$ admits also a $\star$-invariant form if and only if $X\cong X^\delta.$ Notice that if there exists an $\bH$-isomorphism $\tau^\delta_X: (\pi^\delta,X^\delta)\to (\pi,X),$ then $X$ can be lifted to a simple $\bH'$-module, where $\delta$ acts by $\tau^\delta_X.$ In section \ref{sec:LWT}, we will see that when $\bH$ is of geometric type, these isomorphisms admit a canonical normalization. Then the $\bullet$ and $\star$-forms on $X$  are related by
\begin{equation}\label{star-bullet-2}
\langle v_1,v_2\rangle_\star=\langle v_1,\pi(t_{w_0}) \tau^\delta_X(v_2)\rangle_\bullet.
\end{equation}

The above analysis has an important application to the relation between the signatures of the form on $W$-isotypic components of $X$. Since $\delta$ acts by conjugation by $w_0$ on $W$, it is clear that $X^\delta|_W\cong X|_W.$ Suppose $\mu$ is an irreducible $W$-representation, and let $X(\mu)$ denote the $\mu$-isotypic component of $\mu$ in $X$. In particular, $X^\delta(\mu)\cong X(\mu)$.
The pairing (\ref{star-bullet-1}) descends to a $W$-invariant pairing
\begin{equation}\label{star-bullet-types}
\langle~,~\rangle_\star^\mu: X^\delta(\mu)\times X(\mu)\to \bC,\ \langle u,v\rangle_\star^\mu=\langle u, \pi(t_{w_0})v\rangle_\bullet,\ u,v\in X_\mu.
\end{equation}
If $X^\delta\cong X$ as $\bH$-modules, the $\bH$-isomorphism $\tau_X^\delta$ induces isomorphisms $\tau^\delta_X(\mu): X^\delta(\mu)\to X(\mu)$, so composing with $\tau^\delta_X(\mu)$ in (\ref{star-bullet-types}), we find a $W$-invariant pairing on $X(\mu)$. We have proved:

\begin{lemma}\label{l:star-bullet-types}
If $(\pi,X)$ is a simple $\bH$-module admitting a $\bullet$-invariant form then
\begin{enumerate}
\item $X$ admits also a $\star$-invariant form if and only if $X^\delta\cong X$, and in this case,
\item the signatures of the two forms on a $W$-isotypic space $X(\mu),$ $\mu\in\widehat W$, are related by $\pi(t_{w_0})\circ \tau_X^\delta(\mu)$, i.e., by the action of $t_{w_0}\delta$.
\end{enumerate}
\end{lemma}

\subsection{The elements $R_w$} {Let $\C O(V_\bC)$ denote the ring of
rational functions on $V_\bC$, and consider the completion of $\bH$
\begin{equation}
\hat\bH=\bH\otimes_{S(V_\bC)}\C O(V_\bC).
\end{equation}   }
Following \cite{L1,BM3}, we define for
every $\al\in \Pi$ the {element of $\hat\bH$} 

\begin{equation}
R_{s_\al}=t_{s_\al}\frac{\al}{k_\al-\al}-\frac{k_\al}{k_\al-\al}.
\end{equation}
{The reason for the normalization $k_\al -\al$ is so that for
  $k_\al>0,$ the intertwining operator
has no poles when evaluating on a negative weight like $w_0\nu.$ In
that case $k_\al -\al(w_0\nu)>k_\al >0.$
}

If $x\in W$ has a reduced expression $x=s_{\al_1}\cdot
s_{\al_2}\cdot\dots\cdot s_{\al_k},$ set
$$R_x=R_{s_{\al_1}}R_{s_{\al_2}}\dots R_{s_{\al_k}}.$$  
Notice that
$$R_x=\sum_{y\le x}t_y a_y,\ a_y\in \C O(V_\bC)
  \text{ and }a_x=\prod_{x^{-1}\al<0}\frac\al{k_\al-\al}.$$

\begin{lemma}\ 

\begin{enumerate}
\item The element $R_w,$ $w\in W$, does not depend on the choice of
  reduced expression for $w$. 
\item For every {$a\in \C O(V_\bC),\ w\in W$}
\begin{equation}\label{e:aR}
a\cdot R_w=R_w\cdot {w^{-1}(a).}
\end{equation}
\item For every $w\in W$, 
\begin{equation}\label{e:tR}
t_w\cdot R_{w_0}=(-1)^{\ell(w)} R_{w_0} \cdot \delta(t_w).
\end{equation}
\item $R_xR_y=R_{xy}$, $x,y\in W.$
\end{enumerate}
\end{lemma}

\begin{proof}
Claims (1) and (2) are in \cite[Lemma 1.6]{BM3}. For (3), it is sufficient to
verify that when $\al\in \Pi,$ $t_{s_\al}\cdot R_{w_0}=-R_{w_0}
t_{s_{\beta}},$ where $\beta=-w_0(\al).$  Write $w_0=w s_\al=s_\beta
w$. It follows that $R_{w_0}=R_w R_{s_\al}=R_{s_\beta} R_w,$ and
therefore $R_{w_0}R_{s_\al}=R_{s_\beta} R_{w_0}.$ Then $(t_{s_\beta}
\beta-k_\beta)R_{w_0}=R_{w_0}(t_{s_\al} \al-k_\al),$ and since
$k_\al=k_\beta,$ $R_{w_0}t_{s_\al}\al=t_{s_\beta} \beta
R_{w_0}=t_{s_\beta} R_{w_0} w_0(\beta)=-t_{s_\beta} R_{w_0}\al.$
Claim (4)
follows immediately from $R_{s_\al}^2=(t_{s_\al}\al-k_\al)^2\frac 1{k_\al^2-\al^2}=1$.
\end{proof}

\begin{lemma}\label{l:starbullet}
The elements $R_x$ satisfy
\begin{enumerate}
 \item $\displaystyle{R_x^\bullet
   =(-1)^{\ell(x)}R_{x^{-1}}\prod_{x^{-1}\al<0}\frac{k_\al-\al}{k_\al+\al}}$;

\item $\displaystyle{R^\star_x
=(-1)^{\ell(x)}t_{w_0}R_{\delta(x)^{-1}}\left(\prod_{\delta(x)^{-1}\al<0}\frac{k_\al-\al}{k_\al+\al}\right)t_{w_0}}$.
\end{enumerate}
\end{lemma}

\begin{proof} Claim (2) follows from (1) by (\ref{e:relstar}). For
  (1), we need to compute $R_{s_\al}^\bullet.$ We have
  $R_{s_\al}^\bullet=[(t_{s_\al}\al-k_\al)(k_\al-\al)^{-1}]^\bullet=(k_\al-\al)^{-1}
  (\al t_{s_\al}-k_\al)=-(k_\al-\al)^{-1} R_{s_\al}
  (k_\al-\al)=-R_{s_\al} \frac {k_\al-\al}{k_\al+\al}.$

\end{proof}

\subsection{Parabolic subalgebras} Let $\Pi_M$ be a subset of simple
roots of $\Pi$ and $R_M^+$ the positive roots spanned by
$\Pi_M$. Denote by $W_M$ the parabolic subgroup of $W$ 
generated by $\{s_\al:\al\in\Pi_M\}$ and by $w_{0,M}$ the long Weyl
group element in $W_M.$

 Let $\bH_M$ be the subalgebra of $\bH$ generated by $\{t_w: w\in
 W_M\}$ and $S(V_\bC).$ The star operations $\star_M$ and
 $\bullet_M$ as in Definition \ref{d:basicstars} for $\bH_M$ are:
\begin{equation}\label{e:starM}
\begin{aligned}
&t_w^{\star_M}=t_{w^{-1}},\ \ w\in W_M,\quad &\omega^{\star_M}=-t_{w_{0,M}}\cdot \overline{w_{0,M}(\omega)}\cdot t_{w_{0,M}},\ \omega\in V_\bC,\\
&t_{w}^{\bullet_M}=t_{w^{-1}},\ w\in W_M,\quad &\omega^{\bullet_M}=\overline\omega,\ \omega\in V_\bC.\\
\end{aligned}
\end{equation}
As before, from Definition \ref{d:graded}, every element of $\bH$ can be written
uniquely as $h=\sum_{w\in W} t_w a_w$, where $a_w\in S(V_\bC)$.
{Denote
\begin{equation}\label{e:minimal-coset}
\C J_M=\text{ the set of coset representatives of minimal length in }W/W(M);
\end{equation}
recall that in every coset $x W(M)$ there exists a unique element of
minimal length.}
 Then, more generally, every $h\in \bH$ can be
written uniquely as 
\[h=\sum_{w\in \C J_M} t_w m_w,\quad m_w\in\bH_M.\]
Define the $\bC$-linear map
\begin{equation}\label{e:eps}
\ep_M:\bH\to\bH_M,\quad \ep_M(h)=m_1.
\end{equation}
In particular, $\ep_M(h m)=\ep_M(h)m,$ for all $m\in \bH_M.$ It is
also easy to see that
\begin{equation}\label{e:automeps}
\delta(\ep_{M}(h))=\ep_{\delta(M)}(\delta(h)),\ h\in\bH,
\end{equation}
where $\Pi_{\delta(M)}=\delta(\Pi_M)$, and $\delta$ is the automorphism from
Lemma \ref{l:autom}. 
We need
the relation between $\star$ and $\star_M$. 

\begin{proposition}[{\cite[Proposition 1.4]{BM3}}]\label{p:stareps}
  For every $h\in\bH$, $\ep_M(h^\star)=\ep_M(h)^{\star_M}.$ 
\end{proposition}

\begin{corollary}\label{c:buleps}
For every $h\in \bH$, $\ep_{\delta(M)}(t_{w_0} h^\bullet t_{w_0})=\delta(\ep_M(h))^{\star_{\delta(M)}}.$ 
\end{corollary}

\begin{proof}
Since $t_{w_0} h^\bullet t_{w_0}=\delta(h)^\star,$ the claim is immediate
from Proposition \ref{p:stareps}  and (\ref{e:automeps}).
\end{proof}

\subsection{Induced modules}\label{sec:induced} Let $\Pi_M\subset \Pi$ be given, and
consider the subalgebra $\bH_M$ of $\bH$. If $(\sigma,U_\sigma)$ is an
$\bH_M$-module, consider the induced module 
\begin{equation}
X(M,\sigma)=\bH\otimes_{\bH_M}U_\sigma,
\end{equation} 
where $\bH$ acts by left multiplication. The goal is to construct
invariant hermitian forms on $X(M,\sigma)$ provided that $\sigma$
admits such a form for $\bH_M.$ For this, we need to describe the
$\bH$-module structure $\pi_\sigma$ on $X(M,\sigma)$ more explicitly.

A basis for $X(M,\sig)$ is 
\[
\{ t_x\otimes v_i\},\quad  x\in \C J_M,\quad v_i\in \C B(U_\sig),
\]
where $\C B(U_\sig)$ is a basis of $U_\sigma.$

Every $z\in W$ can be written uniquely 
\begin{equation}\label{e:z-decomp}
z=c(z)\cdot m(z),
\end{equation}
where $c(z)$ is the element of $\C J_M$ in the coset $zW(M)$ and
$m(z)\in W(M).$

\begin{lemma}\label{l:action-induced}
The action $\pi_\sigma$ on $X(M,\sigma)$ is given by
\begin{equation}
\begin{aligned}
\pi (t_z)(t_x\otimes v)&=t_{c(zx)}\otimes \sig(m(zx))v;\\
\pi (\omega)(t_x\otimes v)&=t_x\otimes \sig
({x^{-1}}(\omega))v +\sum_{\sbst{\beta>0}{x^{-1}\beta<0}}
(\omega,\beta^\vee)\ t_{c(s_\beta x)}\otimes
\sig(m(s_\beta x))v,\\
\end{aligned}
\end{equation}
for every $z\in W$ and $\omega\in V_\bC.$
\end{lemma}

\begin{proof}
For $z\in W,$
\[
\pi (t_z)(t_x\otimes v)=t_{zx}\otimes
v=t_{c(zx)}\otimes \sig(m(zx))v.  
\]

For {$\omega\in V_\bC,$ }
\[
\pi (\omega)(t_x\otimes v)=\omega t_x\otimes
v.
\]
The claim follows from the relation:
\[
\omega t_x=
t_x {x^{-1}}(\omega) +\sum_{\substack{\beta >0\\ x^{-1}\beta
  <0}} (\omega,\beta^\vee) t_{s_\beta x}.
\]
\end{proof}

\subsection{Action on  the hermitian dual to an induced modules}\label{sec:action-dual-induced} Let
$(\sig,U_\sig)$ be a module for $\bH_M$ as in section
\ref{sec:induced}. Let $(\sigma^\bullet, U_\sigma^h)$ be the
hermitian dual of $(\sigma,U_\sigma)$ and
$(\pi_\sigma^\bullet,X(M,\sigma)^h)$, the hermitian dual of
$(\pi_\sigma, X(M,\sigma))$ with respect to the star operation $\bullet$.
A basis for the hermitian
dual $X(M,\sigma)^h$ of $X(M,\sig)$ is
\begin{equation}
\begin{aligned}
\{ t^h_x\otimes v^h_i\},\text{ where } x\in \C J_M \text{ and }
v_i^h\in U_\sig^h \text{ dual to the basis } \C B(U_\sig)=\{v_i\}.
\end{aligned}
\end{equation}
We calculate the action $\pi_{\sig}^\bullet$ of $\bH$ on
$X(M,\sig)^h.$  For $z\in W,$
\begin{equation}
  \label{eq:1.10.1}
\pi^\bullet (t_z)(t_x^h\otimes v_i^h)(t_y\otimes v_j)=(t_x^h\otimes
v_i^h)(t_{z^{-1}x}\otimes v_j).  
\end{equation}
Then (\ref{eq:1.10.1}) is nonzero if and only if
$c(z^{-1}y)=x,$ so 
\[
z^{-1}y=xm(z^{-1}y),\text{ or equivalently, }zx=ym(z^{-1}y)^{-1}.
\]
We conclude that $m(zx)=m(z^{-1}y)^{-1},$ and so 
\begin{equation}
\pi^\bullet(t_z)(t^h_x\otimes v^h_i)=t^h_{c(zx)}\otimes\sig^\bullet (m(zx))v^h_i.
\end{equation}

For $\omega\in V_\bC,$ 
\[
\pi^\bullet (\omega)(t^h_x\otimes v^h_i)(t_y\otimes v_j)=(t_x^h\otimes
v_i^h)(\omega t_y\otimes v_j).
\]
Using Lemma \ref{l:conjtw}
\[
\omega t_y=t_y {y^{-1}}(\omega) -\sum_{\sbst{\gamma >0}{y\gamma
  <0}} (\omega,y\gamma^\vee) t_{ys_\gamma}=
t_y {y^{-1}}(\omega) +\sum_{\sbst{\beta >0}{y^{-1}\beta
  <0}}(\omega,\beta^\vee) t_{s_\beta y},
\]
we find that the expression is zero unless either $x=y,$ or 
$c(s_\beta y)=x.$ In this latter case,
\[
s_\beta y=x\cdot m(s_\beta y), \text{ equivalently }s_\beta x=ym(s_\beta x),\quad \text{ so }\quad
  m(s_\beta x)=m(s_\beta y)^{-1}.
\]
The conclusion is
\[
\pi^\bullet(\omega)(t_x^h\otimes v^h_i)=t_x^h\otimes \sig^\bullet
({x^{-1}}(\omega))v_i^h {-}
\sum_{\sbst{\beta>0}{c(s_\beta x)^{-1}\beta<0}}
 (\omega,\beta^\vee)\ t_{c(s_\beta x)}
\otimes\sig^\bullet(m(s_\beta x))v_i^h. 
\]
{Notice that since $y\in \C J_M$, if $y^{-1}\beta <0,$ then  in fact
  $y^{-1}\beta\in R^-\setminus R^-_M.$ 
We show that 
\[
c(s_\beta x)^{-1}\beta<0\text{ if and only if }x^{-1}\beta\in
R\setminus R_M.
\]
We have $s_\beta x=ym$ for some $m\in W(M)$. Then $y^{-1}=mx^{-1}s_\beta$, and
$x^{-1}=m^{-1}y^{-1}s_\beta$. 

If $x^{-1}\beta\in R\setminus R_M,$
\[
y^{-1}\beta =m^{-1}x^{-1}(-\beta)\in R^-\setminus R^-_M.
\]
So $y^{-1}\beta < 0.$

If $y^{-1}\beta<0$ then as observed earlier,
$y^{-1}\beta\in R^-\setminus R^-_M,$ so 
\[
x^{-1}\beta=m^{-1}y^{-1}(-\beta)\in m^{-1}(R^+\setminus
R^+_M)=R^+\setminus R^+_M.
\]
In conclusion, we have proved the following formulas for the action
$\pi_\sigma^\bullet.$
\begin{lemma}\label{l:action-dual-induced} The $\bH$-module action on the $\bullet$-hermitian dual
  module $(\pi_\sigma^\bullet,X(M,\sigma)^h)$ is given by:
\begin{equation}
  \label{eq:bulletaction}
  \begin{aligned}
&\pi^\bullet(t_z)(t^h_x\otimes v^h_i)=t^h_{c(zx)}\otimes
\sig^\bullet (m(zx))v^h_i,\\
&\pi^\bullet(\omega)(t_x^h\otimes v^h_i)=t_x^h\otimes\sig^\bullet
({x^{-1}}(\omega))v_i^h {-}\sum_{\sbst{\beta>0}{x^{-1}\beta\in R^+\setminus
  R_M^+}}
(\omega,\beta^\vee)\ t_{c(s_\beta
  x)}\otimes \sig^\bullet(m(s_\beta x))v_i^h. 
  \end{aligned}
\end{equation}
\end{lemma}

}

\subsection{Hermitian dual of an induced module} \label{sec:herm-dual}
Retain the notation from the previous sections. In particular, write
$w_0$ for the long Weyl group element of $W$, $w_{0,M}, w_{0,\delta(M)}$ for
the corresponding long elements in the Levi components, and set
$$w^0_M:=w_0w_{0,\delta(M)}=w_{0,M}w_0.$$ This element is minimal in the cosets
$w_0W_{\delta(M)},$ and $W_Mw_0.$ 

Let $(\sig,U_\sig)$ be an $\bH_M-$module. 
Recall that $(\sig^\bullet, U_\sig^h)$ is the module on the hermitian
dual with respect to the $\bullet$ action. 
\begin{lemma}\label{l:2.6.1}
The map $\phi$ given by
\[
\begin{aligned}
  &\phi(t_m):=t_{(w^0_M)^{-1}mw^0_M},\quad m\in W_M,\\
  &\phi(\omega):=(w^0_M)^{-1}(\omega),\quad \omega\in V_\bC,
\end{aligned}
\]
is an isomorphism between $\bH_M $ and $\bH_{\delta(M)}$ and it 
interchanges $\bullet_M$ with $\bullet _{\delta(M)}.$ 
\end{lemma}

\begin{proof}Straightforward.
\end{proof}

\begin{definition}\label{d:transfer}
In light of Lemma \ref{l:2.6.1}, to each $\bH_M$-module $(\sig,U_\sig)$, we associate the $\bH_{\delta(M)}-$module
$(a\sig,U_{a\sig})$ given by 
\begin{equation}
U_{\Caa\sig}=U_\sig,\text{ and }(\Caa\sig)
(m'):=\sig(w^0_Mm'(w^0_M)^{-1}),\ 
m'\in \bH_{\delta(M)}.
\end{equation}
\end{definition}

\begin{proposition} 
The element 
$x\in W$ is minimal in  $xW_M$ if and only if  $xw^0_M$ is minimal in $xw^0_MW_{\delta(M)}.$
\end{proposition}

\begin{proof} We observe that $w_0(R^\pm_{\delta(M)})=R^\mp_M.$ Then 
\[
w_0w_{0,\delta(M)}(R^+_{\delta(M)})=w_0(R^-_{\delta(M)})=R^+_M.
\]
The claim follows,
\[
x(R^+_M)\subset R^+\text{ if and only if } xw^0_M(R_{\delta(M)})\subset R^+.
\]   
\end{proof}
\begin{corollary}\label{c:M-aM} In the notation of (\ref{e:z-decomp}):
\[
c_{\delta(M)}(xw^0_M)=c_M(x)w^0_M,\qquad m_{\delta(M)}(xw^0_M)=(w^0_M)^{-1}m_M(x)w^0_M,
\]
for every $x\in W.$
\end{corollary}

\begin{theorem}\label{t:herm-dual}
  The map
\[
\Phi(t^h_x\otimes v^h):=t_{xw^0_M}\otimes \Caa v^h
\]
is an $\bH-$equivariant isomorphism between 
$\left(\pi^\bullet_\sigma,X(M,\sig)^h\right)$ and
$\left(\pi_\sigma,X(\delta(M),\Caa\sig^h)\right)$ where the action on $\Ca\sig^h$ is given
by $\bullet_{\delta(M)}.$ 
\end{theorem}
\begin{proof}
 Using Lemma \ref{l:action-dual-induced}, we have
\[
\begin{aligned}
&\Phi[\pi^\bullet(t_z)(t_x^h\otimes v^h)]=\Phi[t^h_{c_M(zx)}\otimes
\sig^\bullet(m(zx))v^h]=t_{c_M(zx)w^0}\otimes
\Caa[\sig^\bullet(m_M(zx))v^h],\\
&\pi(t_z)\Phi[t_x^h\otimes v^h]=\pi(t_z)[t_{xw^0}\otimes \Caa v^h]=
t_{c_{\delta(M)}(zxw^0)}\otimes (\Caa\sig)^\bullet
[(m_{\delta(M)}(zxw^0))\Caa v^h].  
\end{aligned}
\]
Next
\[
\begin{aligned}
\Phi[\pi^\bullet(\omega)(t_x^h\otimes v^h)]&=
t_{xw^0_M}\otimes\ \Caa[\sig^\bullet(x^{-1}\omega)v^h]
{- }\\
{-}&\sum_{\sbst{\beta>0}{x^{-1}\beta\in R^+\setminus R^+_M}}(\omega,\beta^\vee)
t_{c_M(s_\beta x)w^0}\otimes \Caa [\sig^\bullet(m_M(s_\beta x))v^h]\\
\pi(\omega)\Phi[t_x^h\otimes v^h]&=\pi(\omega)[t_{xw^0}\otimes \Caa v^h]=
t_{xw^0_M}\otimes (\Caa\sig)^\bullet((xw^0_M)^{-1}\omega)\Caa v^h+\\
+&
\sum_{\sbst{\gamma>0}{(xw^0)^{-1}\gamma\in R^+\setminus R^+_{\delta(M)}}}
(\omega,\gamma^\vee)~
t_{c_{\delta(M)}(s_\beta xw^0_M)}\otimes (\Caa
\sig)^\bullet(m_{\delta(M)}(s_\beta xw^0_M))\Caa v^h.
\end{aligned}
\]  
The corresponding expressions are equal because of Corollary
  \ref{c:M-aM}, and the fact that $w^0(R^+\backslash R^+_M)=R^-\backslash R^-_{\delta(M)}$.
\end{proof}

\begin{example}\label{e:ps} A particular case of Theorem \ref{t:herm-dual} is that
  of minimal principal series. The hermitian dual $(\pi^\bullet, X(\nu)^h)$ of a minimal principal series
  module identifies with $(\pi,X(w_0\overline\nu))$ via
\[
\Phi(t_x^h\otimes \one_\nu)=t_{xw_0}\otimes \one_{w_0\overline\nu}.
\]
{In particular, this means that $X(\nu)$ admits an invariant $\bullet$
  form if and only if $w_0\overline\nu$ is $W$-conjugate to $\nu,$
  equivalently if $\overline\nu$ is $W$-conjugate to $\nu.$
  Thus, for example, if $w_0$ is not central in $W$,
  $X(\nu)$ does not admit a $\bullet$-form for generic purely
  imaginary values of $\nu.$
}
\end{example}

\subsection{Second form of Frobenius reciprocity} As an application of
Theorem \ref{t:herm-dual}, we obtain the following lemma, which 
is the $\bH$-analogue of the second form of Frobenius reciprocity.

\begin{lemma} If $\bH_M$ is a parabolic subalgebra of $\bH$, $V$ an $\bH$-module and $U$ an $\bH_M$-module, then
\begin{equation}\label{e:Frob2}
\Hom_{\bH_M}[V|_{\bH_M},U]=\Hom_\bH[V,\bH\otimes_{\bH_{\delta(M)}}\Caa(U)].
\end{equation}
\end{lemma}

\begin{proof}
Theorem \ref{t:herm-dual} computed the hermitian dual of a
parabolically induced module. The same exact statement and proof hold
of course for contragredient modules. We use here the same notation
$V^\bullet$ to denote the contragredient (rather than the hermitian
dual) with respect to the involution $\bullet$. We will also use twice the tautological isomorphism
\begin{equation}\label{e:tautology}
\Hom[A,B^\bullet]=\Hom[B,A^\bullet].
\end{equation}
We have:
\begin{equation}
\begin{aligned}
\Hom_{\bH_M}[V,U]&=\Hom_{\bH_M}[V,(U^{\bullet})^\bullet]=\Hom_{\bH_M}[U^\bullet,V^\bullet]\\
&=\Hom_{\bH}[\bH\otimes_{\bH_M}U^\bullet,V^\bullet] \quad\text{(by first Frobenius reciprocity)}\\
&=\Hom_{\bH}[V,(\bH\otimes_{\bH_M}U^\bullet)^\bullet]\\
&=\Hom_{\bH}[V,\bH\otimes_{\bH_{\delta(M)}}\Caa(U)]\quad\text{(by Theorem \ref{t:herm-dual})}.
\end{aligned}
\end{equation}

\end{proof}

\subsection{Sesquilinear Form} A $\bullet$-invariant sesquilinear form on 
$X(M,\sig)$ is equivalent to defining an $\bH$-equivariant  map
\begin{equation}
\C I:(\pi,X(M,\sig))\longrightarrow (\pi^\bullet,X(M,\sig)^h).
\end{equation}
We call $\C I$ hermitian if $\C I^h=\C I$ or equivalently $\C
I(v)(w)=\ovl{\C I(w)(v)}$, for all $v,w\in X(M,\sig).$  
Recall $M, \delta(M),$ and $w_0=w^0_Mw_{0,\delta(M)}=w_{0,M}w^0_M$ with $w^0_M$
minimal in $w_0W_\wti M.$ To simplify notation, write $\wti M=\delta(M),$ and 
\begin{equation}
w^0=w^0_M,\ R^0:=R_{w^0_M}.
\end{equation}
 Furthermore,
\[
\Ad w^0_M:W_{\wti M}\longrightarrow W_M,\qquad \Ca\sig(\wti m)=\sig(w^0_M
\wti m(w^0)^{-1}).
\]
If $x=c_M(x)m_M(x),$ then $xw^0_M=c_M(x)w^0_M (w^0_M)^{-1}m_M(xw^0_M),$
so 
\begin{equation}
c_M(x)w^0_M=c_{\wti M}(xw^0_M),\quad (w^0_M)^{-1}m_M(x)w^0_M=m_{\wti M}(xw^0_M).
\end{equation}

Assume that there is an $\bH_M$-equivariant isomorphim
$$\iota:(\sig,U_\sig)\longrightarrow (\sig^\bullet,U_\sig^h)$$ defining
a $\bullet$-invariant hermitian form on $(\sig,U_\sig)$.  The same
map gives an isomorphism $\iota_{\Caa}:(\Caa\sig,U_\sig)\longrightarrow
(\Caa\sig^\bullet, U_\sig^h).$ 

Write $R^0=\sum t_z m^0_{\wti z}$ with $\wti z$ minimal in $\wti zW(\wti M)$ and $m^0_{\wti z}\in \bH_{\wti M}.$

Define $\C I$ to be the composition of the maps
\[
X(M,\sig)\xrightarrow{L_{R^0}}
  X(\wti M,\Caa\sig)\xrightarrow{\Ind\iota_{\Caa}}
    X(\wti M,\Caa\sig^h)\xrightarrow{\Phi^{-1}} X(M,\sig)^h.
\]
where:
\begin{enumerate}
\item[(i)]\[
L_{R^0}:t_x\otimes v\mapsto t_xR^0\otimes v=\sum
t_{c_{\wti M}(x\wti z)}\otimes a\sig(m_{\wti M}(x\wti z)m^0_{\wti z})v
\]
\item[(ii)] $\Ind\iota_{\Caa}$ maps it to
\[
\Ind\iota_{\Caa}\circ L_{R^0}:t_x\otimes v\mapsto \sum
t_{c_{\wti M}(x\wti z)}\otimes
\iota_{\Caa} [\Caa\sig(m_{\wti M}(x\wti z)m^0_{\wti z})v].
\]
\item[(ii)]Applying $\Phi^{-1}$ we get
\[
\begin{aligned}
\C I:\ &t_x\otimes v\mapsto \sum t_{c_{\wti M}(x\wti z)(w^0)^{-1}}^h\otimes 
\iota_a [a\sig[m_{\wti M}(x\wti z)m^0_{\wti z}v]=\\
=&\sum t_{c_{M}(x\wti z(w^0)^{-1})}^h\otimes \iota_{\Caa} [\Caa\sig(m_{\wti
  M}(x\wti z)m^0_{\wti z})v].
\end{aligned}
\]
\end{enumerate}
Thus
\[
\langle t_x\otimes v_x,t_y\otimes v_y\rangle=
\langle \Caa\sig(m_{\wti M}(x\wti z)m^0_{\wti z})v_x,v_y\rangle
\]
with $c_M(x\wti z(w^0)^{-1})=y.$ This equation gives 
\[
\begin{aligned}
x\wti z(w^0)^{-1}=&y\cdot m_M(x\wti z(w^0)^{-1})\LR
\wti z=x^{-1}ym_M(x\wti z (w^0)^{-1})w^0\LR\wti z=x^{-1}yw^0m_{\wti  M}
(x\wti z)\LR\\ \LR& x^{-1}yw^0=\wti zm_{\wti M}(x\wti z)^{-1}
\LR\wti z=c_{\wti M}(x^{-1}yw^0),\ m_{\wti M}(x\wti z)=m_{\wti M}(x^{-1}yw^0)^{-1}.
\end{aligned}
\]
The final answer is
\[
\langle t_x\otimes v_x,t_y\otimes v_y\rangle =
\langle \Caa\sig [m_{\wti M}(x^{-1}yw^0)^{-1}
m^0_{c_{\wti M}(x^{-1}yw^0)}]v_x,v_y\rangle. 
\]

Compare and contrast this with
\[
\ep_{\wti M}(t_{(w^0)^{-1}}t_{y^{-1}}t_x R^0)=\ep_{\wti M}(\sum
t_{(w^0)^{-1}}t_{y^{-1}}t_xt_{\wti z}m^0_{\wti z} )=\wti m\cdot
m^0_{\wti z}
\]
where $(w^0)^{-1}y^{-1}x\wti z=\wti m$. So 
\[
x^{-1}yw^0=\wti z{\wti m}^{-1}\LR \wti
z=c_{\wti M}(x^{-1}yw^0),\ \wti m=m_{\wti M}(x^{-1}yw^0)^{-1}.
\]

In conclusion, we have proved the following result.

\begin{proposition}\label{p:induced-bulletform} Suppose $(\sig,U_\sig)$ has a $\bullet$-invariant
  hermitian form $\langle~,~\rangle_{\sigma,\bullet}.$ The form 
\[
\langle h_1\otimes v_1,h_2\otimes v_2\rangle_\bullet=\langle \Caa\sig[\ep_{\wti
  M}(t_{(w^0)^{-1}} h_2^\bullet h_1 R^0)]v_1,v_2\rangle_{\sigma,\bullet}
\]
on $X(M,\sigma)$ is $\bullet$-invariant and sesquiliniar.
\end{proposition}

We prove in the next section that the form is also hermitian.

\subsection{Symmetry}\label{sec:symmetry}
The parabolic Hecke subalgebra $\bH_M$ of $\bH$ is attached
to the non-semisimple root system $(V,R_M,V^\vee,R_M^\vee).$ Let $V_M$
be the $\bR$-span of $R_M$ in $V$, $V_M^\vee$ the $\bR$-span of
$R_M^\vee$ in $V^\vee,$ and 
\begin{align*}
V_M^\perp&=\{v\in V: (v,\al^\vee)=0,\text{ for all }\al\in R_M\},\\
 V_M^{\vee,\perp}&=\{v^\vee\in V^\vee: (\al,v^\vee)=0,\text{ for all }\al\in R_M\}.
\end{align*}
Then $V=V_M\oplus V_M^\perp$, $V^\vee=V_M^\vee\oplus
V_M^{\vee,\perp}.$ Let $\bH_M^0$ denote the graded Hecke algebra
attached to the semisimple root system $(V_M,R_M,V^\vee_M,R_M^\vee)$
by Definition \ref{d:graded}. Then there is an algebra isomorphism
\begin{equation}
\bH_M=\bH_M^0\otimes_\bC S(V_M^\perp).
\end{equation}

Assume $\sig=\sig_0\otimes\bC_\nu$, where $\sigma_0$ is an
$\bH_M^0$-module, and 
$\nu\in (V^{\perp,\vee}_M)_\bC$. 
\begin{lemma}
  \label{l:basis}
There is a set $V^{\perp,\vee}_{M,\reg}$ containing an open
set of $V_M^{\perp,\vee}$ such that $\{\C R_x\otimes v_i\}$ with $x$ minimal in
the coset $xW(M)$ and $\{v_i\}$ a basis of $U_\sig$ forms a basis of $X(M,\sig).$
\end{lemma}
\begin{proof}
  This follows from the formula 
\[
\C R_x=t_x\prod_{x^{-1}\al
    <0}\frac{\al}{k_\al +\al}+ \sum_{y<x}t_y m_y^x.
\]
The leading term for $R^0,$ 
$\displaystyle{\sig\left(\prod_{(w_M^0)^{-1}\al <0}\frac{\al}{k_\al +\al} \right)}$, is
invertible for generic $\nu.$ The claim follows from the fact that the
expression of $\C R_x$ is {upper triangular} in the $t_y.$
\end{proof}
\begin{theorem}\label{t:induced-bulletform}
The form in Proposition \ref{p:induced-bulletform} is hermitian, and
therefore, it gives a $\bullet$-invariant hermitian form on the
induced module $X(M,\sigma)$.
\end{theorem}
\begin{proof}
The claim  follows (on $V^{\perp,\vee}_{M,\reg}$ first, and thus always) from the formula
\[
\ep_{\wti M}(t_{(w^0)^{-1}}\C R^\bullet_y\C R_xR^0)=0\ \text{  unless  } x=y.
\]  
\end{proof}

As above, when $\nu\in (V^{\perp,\vee}_{M,\reg})_\bC,$ a basis of
$X(M,\sigma)$ is given by $\{\C R_x\otimes v\}$, where $x$ ranges in
$\C J_M,$ and $v$ ranges over a basis of $\sigma_0.$ In this case, one
obtains a simpler formula for the signature of the $\bullet$-form.

\begin{corollary}\label{c:reg-form}
When $\nu\in (V^{\perp,\vee}_{M,\reg})_\bC,$ the signature of the
$\bullet$ form on $X(M,\sigma),$ $\sigma=\sigma_0\otimes\bC_\nu$, is
given by
\begin{equation}
\langle \C R_x\otimes v_1,\C R_y\otimes v_2\rangle_{\bullet}=\begin{cases} 0,&x\neq y,\\
\displaystyle{ f(\cc(\sigma))\left\langle \sigma\left(\prod_{\al>0, x\al<0} \frac{\al-k_\al}{\al+k_\al}\right)v_1,v_2\right\rangle_{\sigma,\bullet}},&x=y,
\end{cases} 
\end{equation}
where $x,y\in\C J_M,$ $v_1,v_2\in U_\sigma$, and
$f(\cc(\sigma))=\displaystyle{(-1)^{|R^+\setminus R^+_M|} \prod_{\al\in R^+\setminus R^+_M}\frac{\langle
  \al,\cc(\sigma)\rangle}{k_\al+\langle\al,\cc(\sigma)\rangle}}.$
\end{corollary}

\begin{proof}
The first claim follows as in the proof of Theorem
\ref{t:induced-bulletform}. For the second claim, one uses  the formula 
for $\C R_x^\bullet$ in Lemma \ref{l:starbullet}.
\end{proof}

Since the factor $f(\cc(\sigma))$ is common for all $x$, it makes sense to normalize the hermitian form by dividing by it. The resulting form has the property that
\begin{equation}\label{e:normalized}
\langle\C R_1\otimes v_1,\C R_1\otimes v_2\rangle_\bullet=\langle v_1,v_2\rangle_{\sigma,\bullet}.
\end{equation}

\begin{remark}\label{r:Opdam}
In the particular case when $\sigma_0=\triv$ (so that $\sigma$ is the
one-dimesional character $\bC_\nu$) and $\nu$ is large, we recover a
result of Opdam \cite[Theorem 4.1]{O2}. In that case, the induced
module $X(M,\nu)=\bH\otimes_{\bH_M}\bC_\nu$ is $\bA$-semisimple with a
basis given by $\{\C R_x\otimes \one_\nu: x\in \C J_M\}$, and in the
normalization (\ref{e:normalized}), the form is
\begin{equation}\label{e:induced-triv-reg}
\langle \C R_x\otimes\one_\nu,\C R_y\otimes \one_\nu\rangle_{\bullet}=\delta_{x,y}
\displaystyle{\prod_{\al>0, x\al<0} \frac{\langle\al,\nu\rangle-k_\al}{\langle\al,\nu\rangle+k_\al}}
\end{equation}
It is easy to verify that this formula agrees (switching the
between roots and coroots) with the one in
\cite[Theorem 4.1.(4)]{O2}, after taking the scaling factor
$a(\lambda,k)=\prod_{\al>0}(1-k_\al/\wti\lambda(\al^\vee))$ in
the notation therein.
\end{remark}

\subsection{}We have analyzed the construction of induced $\bullet$-invariant forms. The same type of discussion works for $\star$-invariant forms, or otherwise, the result for $\star$-invariant forms can be deduced via formal manipulations as in section \ref{sec:star-bullet}. We only state the result and skip more details. A similar result was obtained in \cite[section 1.8]{BM3}.

\begin{proposition}
Suppose $(\sigma,U_\sigma)$ has a $\star$-invariant hermitian form $\langle~,~\rangle_{\sigma,\star}.$ The pairing
\[\langle h_1\otimes v_1,h_2\otimes v_2\rangle_\star=\langle\Caa\sig[\ep_{\wti
  M}(h_2^\star h_1 R^0)]v_1,v_2\rangle_{\sigma,\star}
\]
on $X(M,\sigma)$ is a hermitian, $\star$-invariant (sesquiliniar) form.
\end{proposition}

\section{Langlands classification and $\bA$-weights}

We use Langlands classification to deduce certain results about the
$\bA$-weights of irreducible $\bH$-modules. As a consequence, we show
that every irreducible $\bH$-module with real central character admits
a $\bullet$-invariant hermitian form.

\subsection{Langlands quotient} Retain the notation from section
\ref{sec:symmetry}. 
The following form of Langlands
classification is proved in \cite{Ev}.

\begin{theorem}\label{t:Langlands-quot}

\begin{enumerate}
\item[(i)] Let $L$ be an irreducible $\bH$-module. Then $L$ is a quotient of $X(M,\nu)=\bH\otimes_{\bH_M}(\sigma\otimes \bC_\nu)$, where $\sigma$ is an irreducible tempered $\bH_{M}^0$-module, and $\nu\in V_M^\perp$ such that $\re\nu$ is dominant, i.e., $(\re\nu,\al)>0$, for all $\al\in\Pi\setminus\Pi_M.$
\item[(ii)] If $\sigma,\nu$ are as in (i), then $\bH\otimes_{\bH_M}(\sigma\otimes \bC_\nu)$ has a unique irreducible quotient $L(\sigma,\nu)$.
\item[(iii)] If $L(\sigma,\nu)\cong L(\sigma',\nu')$, then $M=M',$ $\sigma\cong\sigma'$, and $\nu=\nu'.$
\end{enumerate}
\end{theorem}

We need to review the
construction of $\Pi_M$, $\sigma$ and $\nu$ from $L$.

Let $\{\om_1^\vee,\dots,\om_n^\vee\}$ be the basis of $V^\vee$ consisting of
fundamental coweights, i.e., the basis dual to $\Pi\subset V.$
For every subset $F\subset \{1,2,\dots,n\}$, let 
$$S_F=\{\sum_{j\notin F} c_j\omega_j^\vee-\sum_{i\in F} d_i
\alpha_i^\vee: c_j>0, d_i\ge 0\}\subseteq V^\vee.$$
A lemma of Langlands, cf. \cite[Lemma 2.3]{Ev} says that for every
$v\in V^\vee$, there exists a unique subset $F$ such that $v\in S_F.$
Denote this subset by $F(v).$ If $v=\sum_{j\notin F} c_j\omega_j^\vee-\sum_{i\in F} d_i
\alpha_i^\vee$, then set $$v^0=\sum_{j\notin F} c_j\omega_j^\vee.$$

On $V^\vee$ define the order relation $\ge$ by $v\ge v'$ if
$v-v'\in \bR_{\ge 0}\Phi^{\vee,+}.$
Then, see for example
\cite[Lemma 2.4]{Ev}, 
\begin{equation}
v_1\ge v_2\text{ implies } v_1^0\ge v_2^0.
\end{equation}
Choose $\lambda\in\Omega(L)$ such
that $\re\lambda$ is maximal with respect to $\ge$ among the real
parts of weights of $L$. Then
\begin{equation}\label{Langlands-nu}
\nu=\lambda|_{V_M^{\vee,\perp}},
\end{equation}
and $\sigma$ is an irreducible $\bH_M^0$-module such that
$\sigma\otimes \bC_\nu$ occurs in the restriction of $L$ to
$\bH_M=\bH_{M}^0\otimes S(V_M^\perp).$ Moreover the weights of
$\sigma$ are:
\begin{equation}\label{weights-sigma}
\Omega(\sigma)=\{\lambda'|_{V_M^\vee}: \lambda'\in\Omega(L),\
\lambda'|_{V_M^{\vee,\perp}}=\nu,\ F(\re\lambda')=\Pi_M\}\subset V_M^\vee.
\end{equation}

\subsection{Iwahori-Matsumoto involution} The Iwahori-Matsumoto involution $\tau$ of $\bH$ is defined on
the generators of $\bH$ by:
\begin{equation}
\tau(t_{s_\al})=-t_{s_\al},\ \al\in\Pi,\quad \tau(a)=-a,\ a\in V_\bC.
\end{equation}
It is immediate that this assignment extends to an algebra
automorphism and therefore to a involution, denoted $\tau$ again on
$\bH$-modules. Notice that if $X$ is an $\bH$-module, then
\begin{equation}
\begin{aligned}
&\tau(X)|_W\cong X|_W\otimes\sgn,\\
&\Omega(\tau(X))=-\Omega(X),\quad \tau(X)_{-\lambda}\cong X_{\lambda},\ \lambda\in\Omega(X).
\end{aligned}
\end{equation}

\begin{lemma}\label{IMtempered} Assume $\bH$ is semisimple. 
Suppose $X$ is an irreducible tempered module such that
$\tau(X)$ is also tempered. Then the central character $\chi$ of $X$ is
imaginary, i.e., $\chi\in \sqrt{-1} V$, and $X\cong X(\chi)$.
\end{lemma}

\begin{proof}
Let $\lambda\in\Omega(X)$ be arbitrary. Since $X$ is tempered,
$(\omega,\re\lambda)\le 0$ for all dominant $\omega\in V.$ If $\tau(X)$
is also tempered, $(\omega,-\re\lambda)\le 0$ as well, hence
$(\omega,\re\lambda)=0$ for all $\omega$ dominant in $V$. Thus
$\re\lambda=0$ and so $\chi\in\sqrt{-1}V$, which means $X\cong
X(\chi)$, since at imaginary central character the minimal principal
series is irreducible (\cite{Ch}, see \cite[Theorem 1.3]{O2}).
\end{proof}

\subsection{$\bA$-weights} 
Let $(\pi,X)$ be an irreducible $\bH$-module, and $\Omega(X)\subset
V_\bC^\vee$ the set of $\bA=S(V_\bC)$-weights of $X$. As noted before,
$\Omega(X)\subset W\cdot \cc(X).$ Define the $\bA$-character of $X$ to
be the formal sum:
\begin{equation}
\Theta_\bA(X)=\sum_{\lambda\in\Omega_X}(\dim X_\lambda)~ e^\lambda,
\end{equation}
where $X_\lambda=\{x\in X: \text{ for all }a\in \bA,\
(\pi(a)-\lambda)^nx=0, \text{ for some }n\}$ is the generalized
$\lambda$-weight space. Denote the multiplicity of $\lambda$ in $X$ by
\begin{equation}
m[\lambda:X]:=\dim X_\lambda.
\end{equation}

The following proposition is the graded Hecke algebra analogue of a
result of Casselman for $p$-adic groups, and Evens-Mirkovi\'c
\cite[Theorem 5.5]{EM} for the geometric affine Hecke algebras.

\begin{proposition}\label{p:weightsunique}
Let $X$, $X'$ be two irreducible $\bH$-modules such that
$\Omega(X)=\Omega(X')$. Then $X\cong X'.$
\end{proposition}

\begin{proof}
By hypothesis $\cc(X)=\cc(X')=\chi$. Suppose $X$ (and therefore also $X'$) is not tempered. By Langlands
classification, $X$ is the unique irreducible quotient of
$\bH\otimes_{\bH_M} (\sigma\otimes \bC_\nu)$, where $\Pi_M\subsetneq
\Pi$, $\nu$ and $\Omega(\sigma)$ are uniquely determined by
$\Omega(X).$ Therefore, by induction of $|\Pi|$, the claim follows for
nontempered $X$.

Now assume that $X$ is tempered. We use the Iwahori-Matsumoto
involution and Lemma \ref{IMtempered}: either $\tau(X)$ is not
tempered and since $\Omega(\tau(X))=\Omega(\tau(X'))$, we may finish
as above, or else $X$ (and also $X'$) is the irreducible minimal
principal series $X(\chi)$ with imaginary central character $\chi.$
\end{proof}

As a consequence, we deduce indirectly that every irreducible module
with real central character has a hermitian $\bullet$-invariant form.

\begin{corollary}\label{c:realbullet}
Let $X$ be an irreducible $\bH$-module. Then $X$ admits a
$\bullet$-invariant hermitian form if and only if
$\overline{\Omega(X)}=\Omega(X).$ In particular, if $X$ has real
central character then $X$ admits a $\bullet$-invariant form.
\end{corollary}

\begin{proof}
Since $a^\bullet=\overline a$ for all $a\in S(V_\bC)$, we have
$\Omega(X^\bullet)=\overline{\Omega(X)},$ where $X^\bullet$ is the
$\bullet$-hermitian dual of $X$. The claim follows at once from Proposition \ref{p:weightsunique}.
\end{proof}

\subsection{Linear independence} Proposition \ref{p:weightsunique} says that
$\Theta_\bA(X)$ uniquely determines $X$. We now prove the stronger
statement that $\{\Theta_\bA(X)\}$ is linearly independent.

\begin{lemma}\label{l:inducedweights}
Suppose $\lambda$ is a weight of the irreducible tempered $\bH_M^0$-module
$\sigma$ and $X(\sigma,\nu)$ is a standard Langlands induced module. Then
$$m[\lambda+\nu: L(\sigma,\nu)]=m[\lambda+\nu: X(\sigma,\nu)]=m[\lambda:\sigma].$$
\end{lemma}

\begin{proof}
By the construction of the Langlands quotient $L(\sigma,\nu),$ the restriction of $L(\sigma,\nu)$ to
$\bH_M$ contains the $\bH_M$-module $\sigma\otimes \bC_\nu$, hence
$\Hom_{\bA}[\sigma\otimes \bC_\nu, L(\sigma,\nu)]\neq 0,$ and
therefore $m[\lambda+\nu:L(\sigma,\nu)]\ge
m[\lambda:\sigma].$ Thus, it is sufficient to prove that
$m[\lambda+\nu:X(\sigma,\nu)]=m[\lambda:\sigma].$ 

By \cite[Proposition 6.4]{BM2}, every weight in
$X(\sigma,\nu)/(\sigma\otimes \bC_\nu)$ is of the form
$w(\lambda+\nu),$ where $\lambda$ is a weight of $\sigma$, and $w\neq 1$
ranges over the set $\C J_M$ of minimal length representatives of $W/W_M.$

 We
claim that if $w\neq 1$ is such a representative, then
$w(\lambda+\nu)\neq \lambda'+\nu,$ for every $\lambda,\lambda'$
weights of $\sigma.$ Let $F\subset\{1,2,\dots,n\}$ be such that
$\lambda\in S_F.$ Then, as before, write $\re\lambda=-\sum_{i\in F} d_i\al_i^\vee$,
$d_i\ge 0$ and $\re\nu=\sum_{j\notin F}c_j\omega_j^\vee.$ 

Since $w\al_i^\vee\in R^{\vee,+}$, for all $i\in F,$ we have
$(w\re\lambda,\omega_j^\vee)\le 0$ for every $i\in F,$ $j\notin F.$ On
the other hand, $w\beta_j<\beta_j$ for $j\notin F,$ so
$(w\re\nu,\beta_j)<(\nu,\beta_j)$ for some $j\notin F.$ This implies that
$(w\re(\lambda+\nu),\re\nu)<(\re(\lambda+\nu),\re\nu)=(\re\nu,\re\nu).$

\end{proof}

\begin{theorem}\label{t:A-linindep}
The set $\{\Theta_{\bA}(X)\}$ where $X$ ranges over the set of
(isomorphism classes of) simple $\bH$-modules is $\bZ$-linearly independent.
\end{theorem}

\begin{proof}
Let 
\begin{equation}\label{e:comb1}
\sum_i c_i \Theta_\bA(X_i)=0
\end{equation}
 be a finite linear combination of
$\bA$-characters, where $\{X_i\}$ are distinct simple modules. Without
loss of generality, we may assume that all $X_i$ have the same central
character and moreover, that the central character is real. By Langlands classification, each $X_i$ is the unique
irreducible quotient $L(M_i,\sigma,\nu_i)$ of an induced
$\bH\otimes_{\bH_{M_i}} (\sigma_i\otimes\bC_{\nu_i}).$ 

Find $\lambda$ a weight in the linear combination such that $\lambda$
is maximal with respect to $\ge$ and no other $\lambda'$ occuring in
the linear combination satisfies $(\lambda')^0>\lambda^0,$ with the notation as
in previous subsection.  There exists a unique $F=F(\lambda)$ such
that $\lambda\in S_F$, and write $\Pi_M$ for the subset of $\Pi$
corresponding to $F$, and
$\nu=\lambda|_{V_M^{\vee,\perp}}$ accordingly. Then $\nu=\lambda^0$. Let 
$\sigma_1,\dots,\sigma_k$ irreducible tempered $\bH_M^0$-modules
so that $L(M,\sigma_t,\nu)$ occurs in (\ref{e:comb1}). 

We claim that if $\lambda_j$ is any weight of a $\sigma_l$, $l=1,k,$
then every time  $\lambda_j+\nu$ occurs in (\ref{e:comb1}), it occurs
in a $\Theta_\bA(L(M,\sigma_t,\nu))$ for some $t=1,k.$ To see this, suppose $\lambda_j+\nu$
appears in $L(M',\sigma',\nu').$ Then there exists an extremal weight
$\lambda'$ such that $(\lambda_j+\nu)\le \lambda'$, but then $\nu=(\lambda_j+\nu)^0\le
(\lambda')^0,$ and by assumption $\nu=(\lambda')^0=\nu',$ $M=M'$ and
$\sigma'=\sigma_t$ for some $t$.

Combining this with Lemma \ref{l:inducedweights}, it follows that
(\ref{e:comb1}) implies
\begin{equation}
\sum_{j=1}^k c_j \Theta_\bA(\sigma_j)=0.
\end{equation}
If $\Pi_M\neq \Pi$, we get $c_j=0$ by induction and continue the same
process with the remaining terms in (\ref{e:comb1}). If $\Pi_M=\Pi$,
i.e., the combination involves only $\bA$-characters of irreducible
tempered $\bH$-modules, apply the Iwahori-Matsumoto involution and
conclude as in the proof of Proposition \ref{p:weightsunique}.

\end{proof}

\section{Signature of hermitian forms and lowest $W$-types}

In order to study the signature of $\bullet$-invariant forms, we need
to construct explicitly the forms whose existence is guaranteed by
Corollary \ref{c:realbullet}. We use Langlands classification together
with the explicit induced forms from section \ref{sec:symmetry}. To
conclude certain results about signatures, we also make use of the
geometric classification of $\bH$-modules (for equal parameters).

\subsection{Tempered modules}\label{sec:2.9} 

\begin{lemma}\label{l:tempind}
Let $X$ be an irreducible tempered $\bH$-module. Then $X$
is a submodule of a parabolically induced module $I(M,\sigma\otimes
\bC_\nu)=\bH\otimes_{\bH_M}(\sigma\otimes\bC_\nu)$, where $\sigma$ is a discrete series module of
$\bH_M^0$ and $\nu\in (V_M^{\perp,\vee})_\bC$ with $\re\nu=0.$
\end{lemma}

\begin{proof} Let $\om_i\in V$, $i=1,n,$  denote the fundamental
  weights of the root system. 
For every weight $\lambda\in\Omega(X)$, define 
\begin{equation}
\C F(\lambda)=\{i: (\om_i,\re\lambda)<0\}.
\end{equation}
Since $X$ is assumed tempered, we necessarily have
$(\om_j,\re\lambda)=0$ for all $j\notin \C F(\lambda).$ We assumed
that the root system is semisimple, therefore, 
\[\re\lambda=-\sum_{i\in\C F(\lambda)} d_i\al_i^\vee,\text{ where } d_i>0.\]
Choose now $\lambda\in \Omega(X)$ such that $\C F:=\C F(\lambda)$ is minimal with
respect to set inclusion. Set 
\[
\Pi_M=\{\al_i: i\in \C F\}.\]
Using the decomposition $V^\vee=V_M^\vee\oplus V_M^{\perp,\vee}$, let
$\nu$ be the projection of $\lambda$ onto $(V_M^{\perp,\vee})_\bC.$
Since $\re\lambda\in V_M^\vee$ by construction, it follows that
$\re\nu=0.$

Let $Y$ be an irreducible consituent of the
restriction of $X$ to $\bH_M$, such that $S((V_M^\perp)_\bC)$ acts on
$Y$ by $\nu.$ Then $Y\cong Y^0\otimes \bC_\nu$, where $Y^0$ is an
irreducible $\bH_M^\sem$-module. We claim that $Y^0$ is a discrete
series $\bH_M^0$-module. To see this, let $\mu\in (V_M^\vee)_\bC$
be a $S((V_M)_\bC)$-weight of $Y^0,$ write
$\re\mu=-\sum_{i\in \C F} z_i\al_i^\vee,$
and we want to prove that all $z_i>0.$ The sum $\mu+\nu$ is a weight
of $X$ and $\re(\mu+\nu)=\re(\mu)=-\sum_{i\in \C F}z_i\al_i^\vee,$ in
particular, $z_i\ge 0.$ Notice that if $j\notin \C F,$ then $j\notin
\C F(\mu+\nu),$ hence $\C F(\mu+\nu)\subseteq \C F.$ By the minimality
of $\C F,$ $\C F(\mu+\nu)=\C F,$ and therefore $z_i>0$ for all $i$.
Setting $\sigma=Y^0,$ the lemma is proved.

\end{proof}

The following statement is well-known.

\begin{proposition}\label{p:unittemp} Every irreducible tempered $\bH$-module is $\star$-unitary.
\end{proposition}

\begin{proof}[Sketch of proof]
When the Hecke algebra $\bH$ appears in the representation theory of $p$-adic groups (i.e., it is ``geometric type'' in the sense of Lusztig \cite{L2}), the claim follows from the unitarizability of tempered representations of the $p$-adic group, see for example \cite{BM1}.

For Hecke algebras with arbitrary positive parameters, the statement is  known from \cite{O} in the setting of affine Hecke algebras, together with the fact that Lusztig's reduction from affine to graded affine Hecke algebras \cite{L1} preserves temperedness and unitarity.
\end{proof}

\begin{corollary}\label{c:bullettemp} Every irreducible tempered
  $\bH$-module with real central character admits a
  $\bullet$-invariant hermitian form.
\end{corollary}

\begin{proof}
Let $X$ be an irreducible tempered module. If $X^\delta\cong X,$ then we
can define a $\bullet$-hermitian form, using the $*$-hermitian form
$\langle~,~\rangle_*$ from Proposition \ref{p:unittemp}, as before, by
setting $\langle
x,y\rangle_\bullet=\langle \pi(t_{w_0}) x, \delta(y)\rangle_*,$ $x,y\in
X.$
We claim that $X^\delta\cong X$ for every irreducible tempered
$\bH$-module with real infinitesimal character. For this, we use
that the restriction to $W$ of the set of tempered modules with real
central character is linearly independent in the Grothendieck group of
$W$. When the parameter function $k$ of $\bH$ is geometric in the
sense of Lusztig, this (and more) follows from the geometric
classification, see section \ref{sec:LWT}. For arbitrary positive
parameters $k$, this result is proved in \cite{So}. 

If $X$ is irreducible tempered with real central character, then $
X^\delta$ is also tempered. This is because $\Omega(X^\delta)=-w_0 (\Omega(X))$,
and if $\om_j$ is a fundamental weight, then so is $-w_0(\omega_j)$,
hence the non-positivity conditions for weights are preserved. 

Also $X|_W\cong X^\delta|_W.$ By the
$W$-linear independence mentioned above, $X\cong X^\delta,$ as $\bH$-modules.
\end{proof}

\subsection{Signature at infinity} 
Assume $\langle\nu,\al\rangle >0$ for all $\al\in R^+\setminus R_M^+$ and denote
$$\sigma_t=\sigma_0\otimes\bC_{t\nu},\ t>0.$$
We
consider the signature of the form on $X(M,\sig_t)$ as
$t\to\infty.$ We can use the basis $\{\C R_x\otimes v_i\},$ so that the
form is block-diagonal with respect to $x\in\C J_M.$ 
By Corollary \ref{c:reg-form}, in the diagonal block (of the normalized form) for $\C R_x$, we have
\[
\langle \C R_x\otimes v_1,\C R_x\otimes v_2\rangle_{\bullet,t}=\left\langle \sigma_t\left(\prod_{\al>0,x\al<0}\frac{\al-k_\al}{\al+k_\al}\right)v_1,v_2\right\rangle_{\sigma_t,\bullet}.
\]

As $t\to\infty,$ the expression $\sigma_t\left(\prod_{x\al<0}\frac{\al-k_\al}{\al+k_\al}\right)$ goes to the identity, which means that
\begin{equation}
\lim_{t\to \infty}\langle \C R_x\otimes v_1,\C R_x\otimes v_2\rangle_{\bullet,t}=\langle v_1,v_2\rangle_{\sigma_0,\bullet}.
\end{equation}
We have proved
\begin{theorem}\label{t:signature} The $\bullet-$signature of $X(M,\sig_t)$ at $\infty$ is the induced
signature of the $\bullet-$signature of $(\sig_0,U_{\sig_0}).$ 

\end{theorem}

{
\subsection{Lowest $W$-types}\label{sec:LWT}
 In this section, we assume that the
graded Hecke algebra $\bH$ has equal parameters. More generally,
analogous results hold whenever the parameters of $\bH$ are of
geometric type, in the sense of \cite{L2}.

Suppose $\bH$ is attached to a root system $\Psi$ and constant parameter
function $k.$ Let
$\fg$ be the reductive Lie algebra with root system
$\Psi$. In particular, we identify a Cartan subalgebra $\fh$
of $\fg$ with $V_\bC^\vee,$ so that the roots $R$ live in $\fh^*\cong
V_\bC.$ Let $\C N\subset\fg$ denote the nilpotent cone. Let $G$ be a
complex connected Lie group with Lie algebra $\fg$; for our purposes,
we may choose $G$ to be the adjoint form. If $\C S$ is a subset of
$\fg$, denote by $Z_G(\C S)$ the mutual centralizer in $G$ of the
elements in $\C S$ and $A(\C S)$ the group of components of $Z_G(\C S).$

We summarize the results
from \cite{KL,L2} that we need for signatures. 

One attaches a standard geometric $\bH$-module $X(s,e,\psi)$ to every triple 
\begin{equation}\label{e:geom-params}
(s,e,\psi),\quad s\in \fg\text{ semisimple},\ e\in\C N\text{ such that }
[s,e]=k e, \ \psi\in \widehat A(s,e)_0,
\end{equation}
where $\widehat A(s,e)_0$ is the set of irreducible representations of
$A(s,e)$ which appear in the permutation action on the top cohomology
$H^{\text{top}}(\C B_{e}^s,\bC)$. Here, $\C B_{e}^s$ denotes the variety of
Borel subalgebras of $\fg$ containing $e$ and $s.$ Morever,
\begin{equation}
X(s,e,\psi)\cong X(s',e',\psi')\text{ if and only if } g\cdot
(s,e,\psi)=(s',e',\psi'),\text{ for some }g\in G.
\end{equation}
Consequently, we may assume, without loss of generality, that
$s\in\fh.$ Under the identification $\fh=V_\bC^\vee,$ write
$s=s_0+\sqrt{-1} s_1$ with $s_0,s_1\in V^\vee.$

On the other hand, recall that the Springer correspondence realizes
every irreducible $W$-representation as the $\phi$-isotypic component
\begin{equation}
\mu(e,\phi):=\Hom_{A(e)}[\phi,H^{\text{top}}(\C B_e,\bC)]
\end{equation}
of the top cohomology group of the Springer fiber $\C B_e$. Denote by
$\widehat A(e)_0$ the set of irreducible representations of $A(e)$
which appear in the action on $H^{\text{top}}(\C B_e,\bC).$ Moreover
$\mu(e,\phi)\cong \mu(e',\phi')$ if and only if there exists
$g\in G$ such that $g\cdot (e,\phi)=(e',\phi').$

The inclusion $Z_G(s,e)\to Z_G(e)$ descends to an inclusion $A(s,e)\to
A(e).$ The standard module $X(s,e,\psi)$ has the property that
\begin{equation}
\Hom_W[\mu(e,\phi):X(s,e,\psi)]=\Hom_{A(s,e)}[\psi: \phi|_{A(s,e)}],
\end{equation}
for all $\phi\in\widehat A(e)_0.$
\begin{definition}
 We call $\mu(e,\phi)$ a lowest
$W$-type of $X(s,e,\psi)$ if $\Hom_{A(s,e)}[\psi:\phi|_{A(s,e)}]\neq
0.$ 
\end{definition}

\begin{theorem}[\cite{KL,L2}] \label{t:geom-class}\ 
\begin{enumerate}
\item The standard module $X(s,e,\psi)$ where
  $(s,e,\psi)$ is as in (\ref{e:geom-params}) has a unique composition
  factor $L(s,e,\psi)$ such that $L(s,e,\psi)$ contains every lowest
  $W$-type of $X(s,e,\psi)$ with full multiplicity
  $[\psi:\phi|_{A(s,e)}].$
\item The module $X(s,e,\psi)$ is tempered if and only if $s_0=k h$
  for a Lie triple $(e,h,f)$ of $e$. In this case,
  $X(s,e,\psi)=L(s,e,\psi).$ The module $X(s,e,\psi)$ is a discrete
  series if in addition $e$ is a distinguished nilpotent element.
\end{enumerate}
\end{theorem}

Notice that, in particular, there is a one-to-one correspondence
between tempered $\bH$-modules with real central character and
($G$-conjugacy classes) of pairs $(e,\phi)$ where $e\in\C N$ and
$\phi\in \hat A(e)_0.$ 

According to the parabolic Langlands classification recalled in Theorem \ref{t:Langlands-quot}, for every
irreducible tempered $\bH_M^0$ module $\sigma_0$ and every $\nu\in
V_M^{\vee,\perp}$ such that $\nu$ is dominant, i.e., $(\al,\nu)>0$ for
all $\al\in \Pi\setminus\Pi_M,$ the standard parabolically induced
module
\begin{equation}
X(M,\sigma_0,\nu)=\bH\otimes_{\bH_M}(\sigma_0\otimes \bC_\nu),
\end{equation}
has a unique irreducible quotient $L(M,\sigma_0,\nu).$ 

The relation with the geometric classification is as follows. The
tempered $\bH_M^0$-module $\sigma_0$ is parameterized by a triple
$(s_M,e_M,\psi_M)$. Here $s_M\in (V_M^\vee)_\bC$, $e_M$ is a nilpotent
element in the corresponding Levi subalgebra $\fk m\subset \fg$ and
$\psi_M$ is a representation of $A_M(s_M,e_M)$. Set
\[s=s_M+\nu\in V_\bC=\fh,\ e=e_M.\]
Since $\nu$ commutes with $s_M$ and $e=e_M,$ $A_G(s,e)=A_G(s_M,e_M).$
The embedding $A_M(s_M,e_M)\to A_G(s_M,e_M)=A_G(s,e)$ induces a
surjection $\widehat A_G(s,e)\to \widehat A_M(s_M,e_M)$, and let
$\psi$ the pull-back of $\psi_M.$ Then
\begin{equation}
X(M,\sigma_0,\nu)\cong X(s,e,\psi) \text{ and } L(M,\sigma_0,\nu)\cong L(s,e,\psi).
\end{equation}
Thus, we may speak of the lowest $W$-types of $X(M,\sigma_0,\nu)$ (and
of $L(M,\sigma_0,\nu).$ 
Denote by $\LWT(M,\sigma_0)$ the set of lowest $W$-types of
$X(M,\sigma,\nu)$.  Since $A_G(s,e)$, $s=s_M+\nu$ does
not change for all dominant $\nu$, this set does not change with
$\nu$, hence the notation. Then Theorem \ref{t:geom-class}(1) implies that $L(M,\sigma_0,\nu)$
contains all the $W$-types in $\LWT(M,\sigma_0)$ with full multiplicity.

Moreover, let $\mu_0$ be the unique lowest $W_M$-type of $\sigma_0.$ For every $\mu\in \LWT(M,\sigma_0)$, we have
\begin{equation}
\Hom_{W}[\mu,X(M,\sigma_0,\nu)]=\Hom_{W(M)}[\mu|_{W(M)},\sigma_0]=\Hom_{W(M)}[\mu|_{W(M)},\mu_0].
\end{equation}
The form $\langle~,~\rangle_{\sigma_0,\bullet}$ can be normalized so that it is positive on the $\mu_0$-isotypic component. 
Then Theorem \ref{t:signature} implies the following corollary.

\begin{corollary}\label{c:lwt} 
The $\bullet$-signature  on the lowest $W-$types of the standard module
$X(M,\sig_0,\nu)$ is given by Theorem \ref{t:signature} for
\emph{all} dominant $\nu.$ In particular, the $\bullet$-form on $L(M,\sigma_0,\nu)$ can be
normalized so that the signature is positive definite on all lowest $W$-types.
\end{corollary}

\begin{remark}\label{lwt-star}
Suppose $L=L(M,\sig_0,\nu)$ also carries a $\star$-invariant hermitian form. Since $\nu$ is real, this is the case precisely when $\delta(M)=M$ and $w_0\nu=-\nu.$ (We have $L(M,\sigma_0,\nu)^\delta=L(\delta(M),\sigma_0,\delta(\nu)).$) As in section \ref{sec:star-bullet}, in order to compare $\star$-signatures with $\bullet$-signatures, we need first to choose an isomorphism $\tau_L^\delta: L^\delta\to L.$ It is an empirical fact that always $L$ has one lowest $W$-type that appears with multiplicity $1$, and we normalize $\tau_L^\delta$ to be $+1$ on the isotypic space of this lowest $W$-type.

From Corollary \ref{c:lwt} and Lemma \ref{l:star-bullet-types}, we see that the $\star$-signature on each isotypic space $L(\mu)$ of a lowest $W$-type $\mu$ of $L(M,\sigma_0,\nu)$ is also independent of (dominant) $\nu$. Moreover,  this signature is given by the action of $t_{w_0}\circ\tau_L^\delta$ on $L(\mu).$ In particular, if $w_0$ is central in $W$ (so $\delta=1$) or if $\dim L(\mu)=1$, the $\star$-form can be normalized so that the $\star$-signature on $L(\mu)$ equals $$(-1)^{h(\mu)}\dim L(\mu),$$ where $h(\mu)$ is the lowest degree in which $\mu$ occurs in harmonic polynomials on $V$.
 
In fact, when the root system is simple, the only case when there exists a lowest $W$-type $\mu$ such that $\dim L(\mu)>1$ is as follows. The root system is of type $E_6$ and the standard module is $X(M,\sigma_0,\nu)$, where $M$ is of type $D_4$ and $\sigma_0$ is the subregular discrete series of $D_4.$ In Theorem \ref{t:geom-class}, this corresponds to a nilpotent element $e$ of type $D_4(a_1)$ in $E_6$, whose centralizer has component group $A(e)=S_3.$ The standard module $X(M,\sigma_0,\nu)$ has three lowest $W$-types denoted $80_s$, $90_s$, and $20_s$ with multiplicities $1$, $2$, and $1$, respectively. One can compute the $\star$-form on the two-dimensional isotypic component of $90_s$ and find that the signature is $(1,-1)$, cf. \cite[page 458]{Ci-e6}. 
\end{remark}

}

\section{Jantzen filtration and hermitian Kazhdan-Lusztig polynomials}
\label{sec:Jantzen-filtration}
\subsection{Jantzen filtration}
 We follow \cite[section 3]{Vo}. Let
$E$ be a complex vector space endowed with an analytic family
$\langle~,~\rangle_t$ of hermitian forms, such that
$\langle~,~\rangle_t$ are nondegenerate for $t\neq t_0$, close to $t_0$. The Jantzen filtration of $E$
(\cite{J}) is a filtration of vector subspaces
$$E=E_0\supset E_1\supset E_2\supset\dots\supset E_N=0,$$
defined as follows. For every $n\ge 0,$ $x\in E$ is in $E_n$ if and
only if there exists $\ep>0$ and a polynomial function
$f_x:(t_0-\ep,t_0+\ep)\to E$ with the properties:
\begin{enumerate}
\item[(i)] $f_x(t_0)=x$;
\item[(ii)] $\langle f_x(t),y\rangle_t$ vanishes at least to order $n$
  at $t=t_0.$
\end{enumerate}
Set 
\begin{equation}
\langle x,y\rangle^n=\lim_{t\to t_0} \frac 1{(t-t_0)^n} \langle f_x(t),f_y(t)\rangle_t;
\end{equation}
this definition is independent of $f_x,f_y.$

\begin{theorem}[{Jantzen \cite[5.1]{J}, cf. Vogan \cite[Theorem 3.2,
    Corollary 3.6]{Vo}}]
The pairing $\langle~,~\rangle^n$ is a hermitian form on $E_n$ with
radical $E_{n+1}$. In particular,
\begin{enumerate}
\item[(a)] $\text{Rad}\langle~,~\rangle_0=E_1$;
\item[(b)] $\langle~,~\rangle^n$ is a nondegenerate hermitian form on $E_n/E_{n+1}.$
\end{enumerate}
Suppose $(p_n,q_n)$ is the signature of $\langle~,~\rangle^n$ on
$E_n/E_{n+1}.$ If $(p^+,q^+)$ is the signature of
$\langle~,~\rangle_t$ for $t>t_0$ and $(p^-,q^-)$ is the signature of
$\langle~,~\rangle_t$ for $t<t_0$, then
\begin{enumerate}
\item[(c)] $p^+=p^-+\sum_{n \text{ odd}} p_n-\sum_{n \text{ odd}} q_n$
    and $q^+=q^-+\sum_{n \text{ odd}} q_n-\sum_{n \text{ odd}} p_n$.
\end{enumerate}
\end{theorem}

Let $X=X(M,\sigma,\nu)$ be a standard module as in Theorem
\ref{t:Langlands-quot} with Langlands quotient $\overline
X=L(M,\sigma,\nu).$ Consider a polynomial in $t$ family of
parameters $\nu_t$, such that $\nu_1=\nu$ and $X_t=X(M,\sigma,\nu_t)$
is irreducible for $t\neq 1$ in some small interval centered at $1$.
Suppose $\sigma$ is a tempered module with real central character, and
$\nu$ is real. By Corollary \ref{c:bullettemp}, every $X_t$ admits a
$\bullet$-invariant nondegenerate form $\langle~,~\rangle_{t,\bullet}$
that we
assume, as we may by Corollary \ref{c:lwt}, to be positive definite
the lowest $W$-types of $X_t.$ Notice that the $W$-structure and
lowest $W$-types of $X_t$ are independent on $t$. Therefore, we may
think of the modules $X_t$ as being realized on the same vector space
$E$ with the analytic family of hermitian forms
$\langle~,~\rangle_{t,\bullet}$, and the previous discussion
applies. We have the Jantzen filtration of $X$:
\begin{equation}\label{e:Jantzen}
X=X_0\supset X_1\supset X_2\supset \dots\supset X_N=0,
\end{equation}
with the following properties, cf. \cite[Theorem 3.8]{Vo}:

\begin{enumerate}
\item[(a)] the filtration (\ref{e:Jantzen}) is a filtration by
  $\bH$-modules;
\item[(b)] $X_0/X_1$ is the Langlands quotient $\overline X$;
\item[(c)] The form $\langle ~,~\rangle_{\bullet}^n$ on $X_n/X_{n+1}$
  is nondegenerate and $\bullet$-invariant. Let $(p_n,q_n)$ be its
  signature. If $(p^\pm,q^\pm)$ is the signature of the
  $\langle~,~\rangle_{\bullet,t}$ for $t<1$, respectively
  $t>1$, then $p^+=p^-+\sum_{n \text{ odd}} p_n-\sum_{n \text{ odd}} q_n$
    and $q^+=q^-+\sum_{n \text{ odd}} q_n-\sum_{n \text{ odd}} p_n$.
\end{enumerate}

\subsection{Kazhdan-Lusztig polynomials}
We recalled in Theorem \ref{t:geom-class} the geometric classification
of standard and simple $\bH$-modules. We record now the known results about
the composition factors of a standard module. Retain the notation from
section \ref{sec:LWT}. In particular, let $s\in \fh$ be the semisimple
parameter, and let $\Irr_s\bH$ denote the irreducible  $\bH$-modules
with central character $W\cdot s.$
Denote
\begin{equation}
G(s)=\{g\in G: \Ad(g)s=s\},\quad \fg_1(s)=\{x\in\fg: [s,x]=x\}.
\end{equation}
It is well-known that $G(s)$ acts on $\fg_1(s)$ with finitely many
orbits. Let $\mathcal C(s)$ denote the set of orbits. Theorem
\ref{t:geom-class} can be rephrased as saying that there is a natural
bijection: 
\begin{equation}
\Irr_s\bH\leftrightarrow \{(\CO,\C L): \CO\in \mathcal C(s),\ \C
L\text{ irr. local system of Springer type supported on }\CO\}.
\end{equation}
Let $L(\CO,\C L)$ denote the irreducible $\bH$-module and $X(\CO,\C
L)$ the corresponding standard module. In this setting, the
Kazhdan-Lusztig conjectures take the following form.

\begin{theorem}[{\cite[Theorem 8.5]{L2}}] In the Grothendieck group of
  $\bH$-modules, 
$$X(\CO,\C L)=\sum_{(\CO',\C L')} P_{(\CO,\C L),(\CO',\C L')}(1)~
L(\CO',\C L'),$$
where
\begin{equation}\label{e:KLpoly}
P_{(\CO,\C L),(\CO',\C L')}(q)=\sum_{i\ge 0}[\C L: \C
H^{2i}IC(\overline{\CO'},\C L')|_{\CO}] \cdot q^i;
\end{equation}
here $\C H^\bullet IC(~)$ denote the cohomology groups of the
intersection cohomology complex.
\end{theorem}

The polynomials $P_{(\CO,\C L),(\CO',\C L')}(q)$ can be computed using
the algorithms in \cite{L3}. In fact, \cite{L3} computes the related
$v$-polynomials
\begin{equation}\label{e:Jantzen-poly}
c_{(\CO,\C L),(\CO',\C L')}(v)=v^{\dim\CO'-\dim\CO}\cdot P_{(\CO,\C
  L),(\CO',\C L')}\left(\frac 1{v^2}\right).
\end{equation}
These polynomials enter in the Jantzen conjecture for $\bH$.

\begin{conjecture} [Jantzen conjecture]\label{Jantzen-conj} Let $X_0\supset
  X_1\supset\dots$ be the Jantzen filtration (\ref{e:Jantzen}) of
  $X=X(\CO,\C L).$ 
\begin{enumerate}
\item[(a)] For every $n\ge 0$, the $\bH$-module $X_n/X_{n+1}$ is
  semisimple.
\item[(b)] The multiplicity of the irreducible module $L(\CO,\C L)$ in
  $X_n/X_{n+1}$ equals the coefficient of $v^n$ in the polynomial
  $c_{(\CO,\C L),(\CO',\C L')}(v)$ defined in (\ref{e:Jantzen-poly}),
  or equivalently,
\item[(b')] 
\begin{equation}
P_{(\CO,\C L),(\CO',\C L')}(q)=\sum_{n\ge 0} m_{(\CO,\C L),(\CO',\C
  L')}(n)~ q^{\frac{\dim\CO'-\dim\CO-n}2},
\end{equation}
where $m_{(\CO,\C L),(\CO',\C
  L')}(n)$ denotes the multiplicity of the irreducible module
$L(\CO',\C L')$ in $X_n/X_{n+1}.$
\end{enumerate}
\end{conjecture}

\subsection{Hermitian Kazhdan-Lusztig polynomials} 
As in the previous subsection, let $X=X(\CO,\C L)$ be a standard
module with Jantzen filtration $X=X_0\supset X_1\supset\dots.$ Suppose
in addition that $s$ is real, i.e., $s\in \fh_\bR.$

Let $$\gr X=\bigoplus_{n\ge 0} X_n/X_{n+1}$$
denote the associated graded $\bH$-module. In section
\ref{sec:Jantzen-filtration}, we have defined a nondegenerate
$\bullet$-invariant form $\langle~,~\rangle^n_\bullet$ on each
$X_n/X_{n+1}$. Let $\langle~,~\rangle^X_\bullet$ be the direct sum
form $\bigoplus_{n\ge 0} \langle~,~\rangle^n_\bullet$ on $\gr X.$

By Corollary \ref{c:lwt}, every irreducible module $L(\CO',\C L')$ has
a canonical $\bullet$-invariant form $\langle~,~\rangle^{(\CO',\C
  L')}_\bullet$ which is positive definite on every lowest
$W$-type. Fix such a form for every $L(\CO',\C L').$ Assuming the truth of Conjecture
\ref{Jantzen-conj}(a), the form $\langle~,~\rangle^n_\bullet$ on $X_n/X_{n+1}$
induces a nondegenerate form on the isotypic component of $L(\CO',\C L')$ in $X_n/X_{n+1}$ whose 
signature is $$(p_{(\CO,\C L),(\CO',\C L')}(n),q_{(\CO,\C L),(\CO',\C
  L')}(n));$$ of course, $p_{(\CO,\C L),(\CO',\C L')}(n)+q_{(\CO,\C L),(\CO',\C
  L')}(n)=m_{(\CO,\C L),(\CO',\C
  L')}(n)$. With this notation, we have
\begin{equation}
(X_n/X_{n+1},\langle~,~\rangle^n_\bullet)=\sum_{(\CO',\C L')}
(p_{(\CO,\C L),(\CO',\C L')}(n)-q_{(\CO,\C L),(\CO',\C  L')}(n))
  \left(L(\CO',\C L'),\langle~,~\rangle^{(\CO',\C L')}_\bullet\right).
\end{equation}
\begin{definition}
Analogous to \cite{ALTV}, define the hermitian Kazhdan-Lusztig
polynomials
\begin{equation}\label{e:herm-KL}
P^h_{(\CO,\C L),(\CO',\C L')}(q)=\sum_{n\ge 0}(p_{(\CO,\C L),(\CO',\C
  L')}(n)-q_{(\CO,\C L),(\CO',\C  L')}(n)) ~ q^{\frac{\dim\CO'-\dim\CO-n}2}.
\end{equation}
\end{definition}
From the definition, it is clear that
\begin{equation}
(\gr X, \langle~,~\rangle^X_\bullet)=\sum_{(\CO',\C L')} P^h_{(\CO,\C
  L),(\CO',\C L')}(1) ~\left(L(\CO',\C L'),\langle~,~\rangle^{(\CO',\C L')}_\bullet\right).
\end{equation}

The question is to compute the polynomials $P^h_{(\CO,\C L),(\CO',\C
  L')}(q).$ We make the following conjecture, motivated by the main theorem of \cite{ALTV}. 

\begin{conjecture}\label{main-conj} For every $(\CO,\C L)$, there exists an orientation number $\ep(\CO,\C L)\in \{\pm 1\}$, such that
$$P^h_{(\CO,\C L),(\CO',\C L')}(q)=\ep(\CO,\C L)\ep(\CO',\C L')~ P_{(\CO,\C L),(\CO',\C L')}(-q).$$
\end{conjecture}

In the rest of the section, we present some examples in support of this conjecture and determine the explicit form of the orientation number in some cases. In particular, we prove Conjecture \ref{main-conj} in the case of regular central character, see Proposition \ref{p:conj-reg}.

\subsection{Regular central character}
Let $\bH$ be a graded Hecke algebra with parameter function
$k$. Recall the minimal principal series $X(\nu)$ with real parameter
$\nu\in V^\vee.$ Suppose $\nu$ is dominant, i.e., $(\al,\nu)>0$ for
all $\al\in R^+.$ A basis of $X(\nu)$ is given by the $\bA$-weight
vectors $\{\C R_x\otimes \one_\nu\}_{x\in W}$ from (\ref{eq:basis}),
and every $\bA$-weight space has multiplicity $1$. In particular, this
means that every irreducible subquotient of $X(\nu)$ occurs with
multiplicity $1$.

 If we normalize the form $\langle~,~\rangle_\bullet$ on $X(\nu)$ so that 
$$\langle \C R_1\otimes \one_\nu,\C R_1\otimes\one_\nu\rangle=1,$$
by (\ref{eq:bprod}), we have
\begin{equation}\label{e:form-regular}
\langle\C R_x\otimes\one_\nu,\C R_x\otimes\one_\nu\rangle_\bullet=\prod_{\beta>0, x\beta<0}\frac{(\beta,\nu)-k_\beta}{(\beta,\nu)+k_\beta}.
\end{equation}
In particular, one gets the following well-known result:

\begin{lemma}
If $\nu$ is dominant, $X(\nu)$ is reducible if and only if there exists $\beta>0$ such that $(\beta,\nu)=k_\beta.$
\end{lemma}

Moreover, (\ref{e:form-regular}) allows us to determine
easily the levels of the Jantzen filtration of $X(\nu).$ For every
$x\in W$, set
\begin{equation}\label{e:tau}
\tau(x,\nu)=\{\beta>0: x\beta<0 \text{ and } (\beta,\nu)=k_\beta\}.
\end{equation}

\begin{lemma}
Suppose $\nu$ is dominant. The $n$-th level in the Jantzen filtration (\ref{e:Jantzen})
of $X(\nu)$ is $$X(\nu)_n=\text{span} \{\C R_x\otimes \one_\nu:
\tau(x,\nu)\ge n\},$$
where $\tau(x,\nu)$ is as in (\ref{e:tau}).
\end{lemma}

\begin{proof}
This is immediate from (\ref{e:form-regular}), since the order of zero
of $\langle \C R_x\otimes\one_\nu,\C
R_x\otimes\one_\nu\rangle_\bullet=\tau(x,\nu)$ and the form
$\langle~,~\rangle_\bullet$ is diagonal in the basis $\{\C R_x\otimes\one_\nu\}.$
\end{proof}

\subsection{} Now suppose that the parameter function for the Hecke algebra is
constant $k=1$.
We analyze first the case $\nu=\rho^\vee.$ Consider the one-parameter family
$X(\nu_t)$, $\nu_t=t\rho^\vee$, $t$ close to $1$. For every positive
root $\beta$, the positive integer $(\beta,\rho^\vee)$ is the height of $\beta.$ We have $(\beta,\rho^\vee)=1$ if and only if $\beta$ is a
simple root. Then
\begin{equation}\label{form-rho}
\langle\C R_x\otimes\one_{\nu_t},\C
R_x\otimes\one_{\nu_t}\rangle_\bullet=
 \left(\frac{t-1}{t+1}\right)^{\tau_0(x)}\displaystyle{\prod_{\beta>0,
  x\beta<0, (\beta,\rho^\vee)>1}\frac{t(\beta,\rho^\vee)-1}{t(\beta,\rho^\vee)+1}>0}.
\end{equation}
where 
\begin{equation}
\tau_0(x)=\{\al\text{ simple root}: x\al<0\}.
\end{equation}
This implies that the $n$-th level of the associated graded of the
Jantzen filtration at $\rho^\vee$ is given by
\begin{equation}
X_n/X_{n+1}=\text{span}\{\C R_x\otimes \one_{\rho^\vee}: \tau_0(x)=n\},
\end{equation}
and $n$ ranges from $0$ to $|\Pi|$, the number of simple
roots. The classification of simple $\bH$-modules with central
character $\rho^\vee$ is well-known: there are $2^{|\Pi|}$ simple
$\bH$-modules, one for each subset of the simple roots, and each one
occurs with multiplicity $1$ in $X(\rho^\vee).$ Formula
(\ref{form-rho}) implies that each irreducible module contributes 
$+1$ to the $\bullet$-form in the level of $X(\rho^\vee)$ where it occurs.

One can analyze similarly the Jantzen filtration at $\rho^\vee$ for
every standard module. Notice that the standard modules
at $\rho^\vee$ are precisely of the form $\Ind_{\bH_J}^\bH
(\St\otimes\bC_{\nu_J})$, where $\nu_J=\rho^\vee-\rho^\vee_J.$

\smallskip

This is consistent with the geometric picture at $\rho^\vee.$ There
are $2^{|\Pi|}$ orbits of $G(\rho^\vee)$ on $\fg_1(\rho)$, each orbit is
of the form $\oplus_{\al\in J} \bC^\times\cdot X_\al$, for a unique $J\subset
\Pi$; here $X_\al$ denote root vectors for $\al\in\Pi.$ In particular,
the closure relations of orbits coincide with the inclusion of subsets
$J$, and the KL polynomials are $P_{J,J'}(q)=1$ if $J\subset J'$, and
$0$ otherwise. In conclusion, at central character $\rho^\vee,$ we
have
\begin{equation}
P^h_{J,J'}(q)=P_{J,J'}(q)=\begin{cases}1,&J\subset J',\\0,&\text{otherwise}.\end{cases}
\end{equation}

\subsection{}
Now suppose that $s$ is an arbitrary regular dominant central character. The structure of the composition series at $s$ reduces to a parabolic subalgebra as follows. Let $$\Delta_s=\{\beta\in R^+: (\beta,s)=1\}.$$
Theorem \ref{t:geom-class} implies in this case that the simple $\bH$-modules with central character $s$ are in one-to-one correspondence with $G(s)=H$-orbits on $\fg_1(s)=\{x\in\fg: [s,x]=x\}=\text{span}\{x_\beta: \beta\in\Delta_s\}$, where $x_\beta$ is a root vector for $\beta.$ There exists $w\in W$ such that $w\Delta_S\subset\Pi$, i.e., a subset of simple roots, so denote $w\Delta_s=\Pi_M$, for some Levi subgroup $M$.

Set $s'=w^{-1}s.$ Then $s'=\rho_M^\vee+\nu$, where $(\al,\nu)=0$ for all $\al\in \Pi_M.$ It is equivalent to determine $G(s')=H$-orbits on $\fg_1(s)=\text{span}\{x_\al:\al\in\Pi_M\},$ but this reduces the problem to thecase of composition series at $\rho_M^\vee$ in $\bH_M.$ Thus the orbits are in one-to-one correspondence with
$$\{J\subset\Pi_M\}\leftrightarrow \sum_{\al\in J}\bC^\times\cdot x_\al=:\CO^M(J).$$
Since every $\CO^M(J)$ has smooth closure, as before, all KL polynomials are $0$ or $1$ depending on inclusion $J'\subset J.$

Fix $J\subset\Pi_M$. Suppose that we have a standard module $X(J,\nu_J)=\Ind_{\bH_J}^\bH(\St_J\otimes\bC_{\nu_J})$ with Langlands quotient $L(J,\nu_J).$ We want to know the level and orientation number of $L(J,\nu_J).$ Since all $\bA$-weights have multiplicity one, $L(J,\nu_J)$ is uniquely determined by the $\bA$-weight $$\lambda_{L(J,\nu_J)}:=-\rho_J^\vee+\nu_J;$$
here $\nu_J$ is dominant with respect to $\Pi\setminus J.$
Inside the minimal principal series $X(s)$, the $\bA$-weight vector with weight $-\rho_J+\nu_J$ is of the form $\C R_x\otimes \one_s$. Since $\C R_x\otimes \one_s$ has $\bA$-weight $xs,$ it follows that
\begin{equation}\label{e:x}
xs=\lambda_{L(J,\nu_J)}.
\end{equation}
By (\ref{e:form-regular}), the form on $\C R_x\otimes \one_s$ is $$\prod_{\beta>0, x\beta<0}\frac{(\beta,s)-1}{(\beta,s)+1}.$$ The contribution of $\C R_x\otimes \one_s$ to the hermitian form in the associated graded module for $X(s)$ is obtained by replacing $s$ with $st$, $0<t<1$, and taking $\lim_{t\to 1}.$ The sign is 
\begin{equation}\label{e:orientation}
\ep(L(J,\nu_J)):=(-1)^{\ell_0(x)},\text{ where } \ell_0(x)=\#\{\beta>0: 0<(\beta,s)<1\text{ and } x\beta<0\},
\end{equation}
or equivalently,
\begin{equation}
\ell_0(x)=\#\{\beta>0: x\beta<0\text{ and } 0<(x\beta,\lambda_{L(J,\nu_J)})<1\}
\end{equation}
In order to establish the truth of Conjecture \ref{main-conj} at regular central character, it remains to verify that the normalization of $\bullet$-form on $L(J,\nu_J)$ is given by the requirement that $\C R_x\otimes \one_s$ be positive. This is indeed the case as follows. The canonical $\bullet$-form on a simple $\bH$-module is normalized so that it is positive definite on all $W$-types. For $L(J,\nu_J)$ this is equivalent with the normalization which as $\nu_J\to\infty$ has the form positive definite on all of $L(J,\nu_J).$ But by Corollary \ref{c:reg-form}, this is the normalization where the $\bA$-weight vector corresponding to the leading weight $\lambda_{L(J,\nu_J)}=-\rho_J^\vee+\nu_J$ is positive. Thus:

\begin{proposition}\label{p:conj-reg}
Conjecture \ref{main-conj} holds in the case of regular central character with the orientation numbers given by (\ref{e:orientation}).
\end{proposition}

\subsection{Subregular orbit in $B_2$} Consider the semisimple element
$s=(1,0)$ in type $B_2.$ There are three $G(s)$-orbits in $\fg_1(s)$,
which we denote by $0$, $A_1$, and $\wti A_1$ (the notation is
compatible with the labeling of their $G$-saturations). The orbits
have dimension $0$, $2$, and $3$, respectively, and the closure
ordering is the obvious total order. The local systems that enter are
trivial for $0$ and $A_1$, so we drop them from notation, and there are two local systems $\C
L_\triv$ and $\C L_\sgn$ for $\wti A_1.$ The matrix of polynomials
$P$, computed in \cite{Ci} using the algorithms of Lusztig \cite{L3}, is in Table \ref{t:B2}.

\begin{table}[h]\label{t:B2}
\caption{KL polynomials: $B_2,\ s=(1,0)$.}
\begin{tabular}{|c||c|c|c|c|}
\hline
Dim. &$0$ &$2$ &$3$ &$3$\\
\hline
Orbits &$0$ &$A_1$ &$(\wti A_1,\C L_\triv)$ &$(\wti A_1,\C L_\sgn)$\\

\hline
\hline
$0$ &$1$ &$1$ &$1$ &$q$\\
\hline
$A_1$ &$0$ &$1$ &$1$ &$0$\\
\hline
$(\wti A_1,\C L_\triv)$ &$0$ &$0$ &$1$ &$0$\\
\hline
$(\wti A_1,\C L_\sgn)$ &$0$ &$0$ &$0$ &$1$\\
\hline

\end{tabular}
\end{table}

We only need to compute the Jantzen filtration and signatures for
$X(A_1)$ and $X(0).$ For this, we do a computation with the
intertwining operators and the $W$-structure of standard
modules. There are $5$ $W$-types, with the notation in terms of
bipartitions as in \cite{Ca}. The $W$-structure of the standard and the irreducible
modules at $s=(1,0)$ is as follows (the $*$ indicates the lowest $W$-type):
\begin{equation}
\begin{aligned}
&X(0)=\bC[W(B_2)], &L(0)=2\times 0^* + 1\times 1;\\
&X(A_1)=11\times 0+1\times 1+0\times 11, &L(A_1)=11\times 0^*;\\
&X(\wti A_1,\C L_\triv)=L(\wti A_1,\C L_\triv)=1\times 1^*+0\times
11;\\
&X(\wti A_1,\C L_\sgn)=L(\wti A_1,\C L_\sgn)=0\times 2^*.\\
\end{aligned}
\end{equation}

For the case $A_1$, we consider the induced module
$X(A_1,(-1/2+\nu,1/2+\nu))=\Ind_{A_1}^{B_2}(\St\otimes\bC_\nu)$,
$\nu>0,$ whose central character is $(-1/2+\nu,1/2+\nu)$. A direct
calculation with the intertwining operator shows that at $\nu=1/2$,
the Jantzen filtration is given by $L(A_1)$ at level $0$,
and $L(\wti A_1,\C L_\triv)$ at level $1$. The signature of the
$\bullet$-form on each $W$-type for $0\le\nu<1/2$ is given by the
parity of the lowest harmonic degree, and thus it is $+$ for $11\times
0$ and $-$ for $1\times 1.$ The normalization of the $\bullet$-forms
implies then that at $\nu=1/2$, the forms on level $1$ are related by:
\begin{equation}
(X(A_1)_1,\langle~,~\rangle^1_\bullet)=(L(\wti A_1,\C
L_\triv),\langle~,~\rangle^{(\wti A_1,\C L_\triv)}_\bullet),
\end{equation}
and thus $P^h_{(A_1),(\wti A_1,\C L_\triv)}(q)= q^{\frac{3-2-1}2}=1.$

For the case $0$, we consider the minimal principal series
$X(\nu_1,\nu_2)$, $0=\nu_2\le \nu_1\le 1.$ The levels of the Jantzen
filtration at $(1,0)$ are given by the order of zeros of the
intertwining operator as follows: $L(0)$ in level $0$, $L(\wti A_1,\C
L_\sgn)$ at level $1$, $L(A_1)$ at level $2$, and $L(\wti A_1,\C
L_\triv)$ at level $3$. Using again that the signature of $W$-types
for $0\le\nu_1<1$ is given by the parity of the lowest harmonic
degree, we see that forms on levels $1$--$3$ are related by:
\begin{equation}
\begin{aligned}
(X(0)_1/X(0)_2,\langle~,~\rangle^1_\bullet)&=-(L(\wti A_1,\C
L_\sgn),\langle~,~\rangle^{(\wti A_1,\C L_\sgn)}_\bullet);\\
(X(0)_2/X(0)_3,\langle~,~\rangle^2_\bullet)&=(L(A_1),\langle~,~\rangle^{A_1}_\bullet);\\
(X(0)_3,\langle~,~\rangle^3_\bullet)&=(L(\wti A_1,\C
L_\triv),\langle~,~\rangle^{(\wti A_1,\C L_\triv)}_\bullet).\\
\end{aligned}
\end{equation}
Thus, $P^h_{(0),(A_1)}=q^{\frac{2-0-2}2}=1,$
$P^h_{(0),(\wti A_1,\C L_\triv)}= q^{\frac{3-0-3}2}=1$, and
$P^h_{(0),(\wti A_1,\C L_\sgn)}= (-1) q^{\frac{3-0-1}2}=-q.$

\smallskip

In conclusion, for the subregular $s$ in $B_2$, $P^h_{(\C O,\C L),(\C O'\C
  L')}(q)= P_{(\C O,\C L),(\C O'\C
  L')}(-q)$.

\subsection{Subregular orbit in $G_2$}
We choose simple roots for $G_2$: $\al_s=\frac 13(2,-1,-1)$ and
$\al_l=(-1,1,0)$ and fundamental coweights $\om_1^\vee=(1,1,-2)$ and
$\om_2^\vee=(0,1,-1).$ Let $s_1$ and $s_2$ be the simple reflections
corresponding to $\al_s$ and $\al_l$, respectively. There are $6$ irreducible Weyl group
representations, which we label as $1_1$ (the trivial), $1_2$ (the
sign), $1_3$ ($s_1=1,$ $s_2=-1$), $1_4$ ($s_1=-1$, $s_2=1$), $2_1$
(the reflection representation), and $2_2=2_1\otimes 1_3.$

Let $s$ be one half of a neutral element for the subregular nilpotent
orbit in $G_2.$ In our coordinates, we choose $s=(0,1,-1).$

There are four $G(s)$-orbits on $\fg_1(s)$, labeled $0$, $A_1^l$,
$A_1^s$, $G_2(a_1)$, of dimensions $0$, $2$, $3$, and $4$,
respectively. The local systems that enter for $0$, $A_1^l$, $A_1^s$
are trivial, but there are two local systems of Springer type for
$G_2(a_1)$, that we denote $\C L_\triv$ and $\C L_\refl.$ The matrix of polynomials
$P$, computed in \cite{Ci}, is in Table \ref{t:G2}.

\begin{table}[h]\label{t:G2}
\caption{KL polynomials: $G_2,$ subregular $s$.}
\begin{tabular}{|c||c|c|c|c|c|}
\hline
Dim. &$0$ &$2$ &$3$ &$4$ &$4$\\
\hline
Orbits &$0$ &$A_1^l$ &$A_1^s$ &$(G_2(a_1),\C L_\triv)$ &$(G_2(a_1),\C L_\refl)$\\

\hline
\hline
$0$ &$1$ &$1$ &$q+1$ &$1$ &$q$\\
\hline
$A_1^l$ &$0$ &$1$ &$1$ &$1$ &$0$\\
\hline
$A_1^s$ &$0$ &$0$ &$1$ &$1$ &$1$\\
\hline
$(G_2(a_1),\C L_\triv)$ &$0$ &$0$ &$0$ &$1$ &$0$\\
\hline
$(G_2(a_1),\C L_\refl)$ &$0$ &$0$ &$0$ &$0$ &$1$\\
\hline

\end{tabular}
\end{table}

The $W$-structure of the standard and the irreducible
modules at $s=(0,1,-1)$ is as follows (the $*$ indicates the lowest $W$-type):
\begin{equation}
\begin{aligned}
&X(0)=\bC[W(G_2)], &L(0)=1_1^* + 2_1;\\
&X(A_1^l)=1_3^*+2_1+2_2+1_2, &L(A_1^l)=1_3^*;\\
&X(A_1^s)=2_2^*+2_1+1_4+1_2, &L(A_1^s)=2_2^*;\\
&X(G_2(a_1),\C L_\triv)=L(G_2(a_1),\C L_\triv)=2_1^*+1_2;\\
&X(G_2(a_1),\C L_\refl)=L(G_2(a_1),\C L_\refl)=1_4^*.\\
\end{aligned}
\end{equation}

For the case $A_1^s$, we consider the standard induced module
$\Ind_{A_1^s}^{G_2}(\St\otimes\bC_\nu)$, of central character $-\frac
12(2,-1,-1)+\nu(0,1,-1).$ The relevant reducibility point is
$\nu=\frac 12.$ We can analyze the Jantzen filtration and signature of the forms in the same
way as for $B_2$ and find:
\begin{equation}
\begin{aligned}
&(X(A_1^s)_0/X(A_1^s)_1,\langle~,~\rangle^0_\bullet)=(L(A_1^s),\langle~,~\rangle^{(A_1^s)}_\bullet);\\
&(X(A_1^s)_1,\langle~,~\rangle^1_\bullet)=(L(G_2(a_1),\C
L_\triv),\langle~,~\rangle^{(G_2(a_1),\C L_\triv)}_\bullet) +(L(G_2(a_1),\C
L_\refl),\langle~,~\rangle^{(G_2(a_1),\C L_\refl)}_\bullet),
\end{aligned}
\end{equation}
and so $P^h_{(A_1^s),(G_2(a_1),\C L_\triv)}(q) =P^h_{(A_1^s),(G_2(a_1),\C L_\refl)}(q)=q^{\frac{4-3-1}2}=1.$

For the case $A_1^l$, we consider the standard induced module
$\Ind_{A_1^l}^{G_2}(\St\otimes \bC_\nu),$ of central character $-\frac
12(-1,1,0)+\nu(1,1,-2).$ The relevant reducibility point is $\nu=\frac
12$, where we find:
\begin{equation}
\begin{aligned}
(X(A_1^l)_0/X(A_1^l)_1,\langle~,~\rangle^0_\bullet)&=(L(A_1^l),\langle~,~\rangle^{(A_1^l)}_\bullet);\\
(X(A_1^l)_1/X(A_1^l)_2,\langle~,~\rangle^1_\bullet)&=(L(A_1^s),\langle~,~\rangle^{(A_1^s)}_\bullet);\\
(X(A_1^l)_2,\langle~,~\rangle^2_\bullet)&=(L(G_2(a_1),\C
L_\triv),\langle~,~\rangle^{(G_2(a_1),\C L_\triv)}_\bullet),
\end{aligned}
\end{equation}
and so and so $P^h_{(A_1^l),(A_1^s)}(q)=q^{\frac{3-2-1}2}=1$,
and so $P^h_{(A_1^l),(G_2(a_1),\C L_\triv)}(q) =q^{\frac{4-2-2}2}=1$,
and $P^h_{(A_1^l),(G_2(a_1),\C L_\refl)}(q)=0.$

Finally, we have the case $0$, where we consider the minimal principal
series with central character $\nu(0,1,-1).$ The relevant reducibility
point is $\nu=1.$ We compute the Janzten filtration and the signature
of the forms using the normalized long intertwining operator on
$W$-types. The only case where more care is needed is the $W$-type
$2_2$ appearing with multiplicity $2$ which corresponds to the factor
$L(A_1^s).$ The $2\times 2$ matrix giving the operator on this
isotypic space has determinant $\frac{(\frac 12-\nu)(1-\nu)^4}{(\frac
  12+\nu)(1+\nu)^4}$ and trace
$\frac{(1-\nu)(1+3\nu^2)}{(1+\nu)^3(\frac 12+\nu)}.$ In particular,
this implies that one copy of $2_2$ (and hence of $L(A_1^s)$) occurs
in level $1$ of the Jantzen filtration and the other copy in level
$3$. For the signatures, we analyze the eigenvalues. We find the
following structure of the filtration together with signatures:
\begin{equation}
\begin{aligned}
(X(0)_0/X(0)_1,\langle~,~\rangle_\bullet^0)&=(L(0),\langle~,~\rangle^{(0)}_\bullet);\\
(X(0)_1/X(0)_2,\langle~,~\rangle_\bullet^1)&=-(L(A_1^s),\langle~,~\rangle^{(A_1^s)}_\bullet);\\
(X(0)_2/X(0)_3,\langle~,~\rangle_\bullet^2)&=(L(A_1^l),\langle~,~\rangle^{(A_1^l)}_\bullet)-(L(G_2(a_1),\C
L_\refl),\langle~,~\rangle^{(G_2(a_1),\C L_\refl)}_\bullet);\\
(X(0)_3/X(0)_4,\langle~,~\rangle_\bullet^3)&=(L(A_1^s),\langle~,~\rangle^{(A_1^s)}_\bullet);\\
(X(0)_4,\langle~,~\rangle_\bullet^4)&=(L(G_2(a_1),\C
L_\triv),\langle~,~\rangle^{(G_2(a_1),\C L_\triv)}_\bullet).\\
\end{aligned}
\end{equation}
The corresponding hermitian KL polynomials are:
$P^h_{(0),(A_1^l)}=q^{\frac{2-0-2}2}=1$,
$P^h_{(0),(A_1^s)}=(-1)q^{\frac{3-0-1}2}+ q^{\frac{3-0-3}2}=1-q$,
$P^h_{(0),(G_2(a_1),\C L_\triv)}=q^{\frac{4-0-4}2}=1$,
$P^h_{(0),(G_2(a_1),\C L_\refl)}=(-1) q^{\frac{4-0-2}2}=-q.$

\smallskip

In conclusion, for the subregular $s$ in $G_2$, $P^h_{(\C O,\C L),(\C O'\C
  L')}(q)=P_{(\C O,\C L),(\C O'\C
  L')}(-q)$, and Conjecture \ref{main-conj} is verified in this case.

\ifx\undefined\bysame
\newcommand{\bysame}{\leavevmode\hbox to3em{\hrulefill}\,}
\fi

\end{document}